\documentclass[final]{siamltex}
	\usepackage{graphicx}
    \usepackage{cancel}
	\usepackage{amsmath}
	\usepackage{amsfonts}
    \usepackage{nicefrac}
    \usepackage{longtable}
    \newcommand{\Dh}{\Delta_h}
    \newcommand{\nabh}{\nabla_{\! h}}

    \newcommand{\hf}{\nicefrac{1}{2}}
    \newcommand{\nrm}[1]{\left\| #1 \right\|}
    
    \newcommand{\cip}[2]{\left( #1 \middle| #2 \right)}
    \newcommand{\eipns}[2]{\left[ #1 \middle\| #2 \right]_{\rm ns}}
    \newcommand{\eipew}[2]{\left[ #1 \middle\| #2 \right]_{\rm ew}}
    
    \newcommand{\ciptwo}[2]{\left( #1 \middle\| #2 \right)}
    
    \def\x{\mbox{\boldmath $x$}}
        
    \newtheorem{rem}[theorem]{Remark}

\begin{document}

	\title{Convergence Analysis of a Second Order Convex Splitting Scheme for the Modified Phase Field Crystal Equation}
	
	\author{
A. Baskaran\thanks{Department of Mathematics; The University of California; Irvine, CA, USA ({\tt baskaran@math.uci.edu})}
	\and
J.S. Lowengrub\thanks{Department of Mathematics; The University of California; Irvine, CA, USA ({\tt lowengrb@math.uci.edu})}	
	\and
C. Wang\thanks{Mathematics Department; The University of Massachusetts; North Dartmouth, MA, USA ({\tt cwang1@umassd.edu})}
	\and
S.M. Wise\thanks{Mathematics Department; The University of Tennessee; Knoxville, TN, USA ({\tt swise@math.utk.edu})}
	}

	\maketitle

	\begin{abstract}
In this paper we provide a detailed convergence analysis for an unconditionally energy stable, second-order accurate convex splitting scheme for the Modified Phase Field Crystal equation, a generalized damped wave equation for which the usual Phase Field Crystal equation is a special degenerate case.  The fully discrete, fully second-order finite difference scheme in question was derived in a recent work \cite{baskaran12}. 
An introduction of a new variable $\psi$, corresponding to the temporal derivative of the phase variable $\phi$, could bring an accuracy reduction in the formal consistency estimate, because of the hyperbolic nature of the equation. A higher order 
consistency analysis by an asymptotic expansion is performed to overcome this difficulty. In turn, second order convergence in both time and space is established in a discrete $L^\infty \left(0,T; H^3\right)$ norm.  
	\end{abstract}



\pagestyle{myheadings}
\thispagestyle{plain}
\markboth{A. BASKARAN,  J.S. LOWENGRUB, C. WANG, AND S.M. WISE}{SECOND ORDER SCHEME FOR THE MPFC EQUATION}

	\section{Introduction}

The modified phase field crystal (MPFC) equation is given by~\cite{stefanovic06}
	\begin{equation}
\beta\partial_{tt}\phi+\partial_t\phi = \Delta\left(\phi^3 + \alpha \phi +2\Delta\phi +\Delta^2\phi \right),
	\label{MPFC-eq}
	\end{equation}
where $\beta\ge 0$ and $\alpha>0$. Equation~(\ref{MPFC-eq}) is a generalized damped wave equation. The parabolic phase field crystal (PFC) equation is recovered in the degenerate case when $\beta=0$.  See~\cite{baskaran12, stefanovic06, stefanovic09, wang11a, wang10b} and references therein for the physical motivation for the MPFC equation.  The existence and uniqueness of global smooth solutions of the MPFC equation was established in our recent article~\cite{wang10b}, assuming that the initial data are smooth.  Very recently, we devised and implemented a second-order convex splitting scheme for the MPFC equation~\cite{baskaran12}.  The solver for the discrete equations was based on a nearly optimally efficient nonlinear multigrid method.  While we proved \emph{a priori} unconditional stability and  unconditional solvability results for the scheme, we did not perform a convergence analysis.  The goal of this paper is to provide a detailed convergence analysis of the second order convex-splitting scheme for MPFC equation (\ref{MPFC-eq}) proposed in \cite{baskaran12}.  To our knowledge no second order convergence analysis exists for scheme for either the PFC or the MPFC equation.

Because of the close relationship between the MPFC and PFC models, methods for the latter equation can be adapted and applied to the former.  See, for example, \cite{backofen07,cheng08,elder04,hu09,mellenthin08,wise09}  for some recent approximation methods specifically for the PFC model.  Methods specifically designed for the MPFC equation can be found in~\cite{baskaran12,galenko11,stefanovic09,wang10b,wang11a}.   Stefanovic~\emph{et al.},~\cite{stefanovic09} employed a semi-implicit finite difference discretization, with a multigrid algorithm for solving the algebraic equations.  They provide no numerical analysis for their scheme, which is significantly different from schemes we propose and analyze.  The MPFC scheme in~\cite{galenko11} is more or less the same as the first-order convex-splitting that we devised earlier in~\cite{wang10b, wang11a}.

The MPFC equation (\ref{MPFC-eq}) may be viewed as a perturbed gradient flow with respect to an energy.  Specifically, consider a dimensionless spatial energy of the form~\cite{elder02, swift77}
	\begin{equation}
E(\phi) = \int_\Omega\left\{\frac{1}{4}\phi^4+ \frac{\alpha}{2}\phi^2 -\left|\nabla\phi\right|^2 +\frac{1}{2}\left(\Delta\phi\right)^2 \right\}d{\bf x}  ,
	\label{energy}
	\end{equation} 
where $\phi:\Omega \subset \mathbb{R}^2 \to \mathbb{R}$ is the ``atom" density field, and $\alpha>0$ is a constant.  Suppose that $\Omega = (0,L_x)\times(0,L_y)$ and $\phi$ is periodic on $\Omega$.  Define $\mu$ to be the chemical potential with respect to $E$:
	\begin{equation}
\mu := \delta_\phi E = \phi^3 + ( 1 - \epsilon) \phi +2\Delta\phi +\Delta^2\phi  ,
	\label{chem-pot}
	\end{equation}
where $\delta_\phi E$ denotes the variational derivative with respect to $\phi$.  Clearly, the MPFC equation may be redefined as
	\begin{equation}
\beta\partial_{tt}\phi+\partial_t\phi = \Delta\mu  ,
	\label{dyn-conserve}
	\end{equation}
where $\beta\ge 0 $.  As mentioned, when $\beta =  0$ the PFC equation is recovered.  Herein we will restrict ourselves to the case that $\beta >0$ to avoid degeneracy. See the discussion in~\cite{wang10b} for some equations in the literature that are closely related to (\ref{MPFC-eq}). 

First, note that the energy (\ref{energy}) is not necessarily non-increasing in time along the solution trajectories of Eq.~(\ref{dyn-conserve}). However, solutions of the MPFC equation do dissipate a pseudo energy, as we show momentarily.  Also observe that Eq.~(\ref{dyn-conserve}) is not precisely a mass conservation equation due to the term $\beta\partial_{tt}\phi$.  However, it is easy to show that if $\displaystyle{\int_\Omega \partial_t\phi({\bf x},0) \, d{\bf x} = 0}$, then  $\displaystyle{\int_\Omega \partial_t\phi\, d{\bf x} = 0}$ for all time~\cite{wang10b,wang11a}.  Herein we assume $\partial_t\phi({\bf x},0) \equiv 0$, for simplicity, which trivially satisfies the condition for mass conservation.

We now recast the MPFC equation (\ref{dyn-conserve}) as the following system of equations:
	\begin{equation}
\beta\partial_t\psi =  \Delta\mu - \psi  , \quad \partial_t\phi = \psi  .
	\label{first-order-system}
	\end{equation}
And we introduce the pseudo energy
	\begin{equation}
{\mathcal E}(\phi,\psi) := E(\phi) + \frac{\beta}{2} \nrm{\psi}^2_{H^{-1}} .
	\label{energy-plus-time}
	\end{equation} 
See~\cite{wang11a,wang10b} for precise definitions of the $H^{-1}$ inner product and norm.  For well-definedness of the $H^{-1}$  norm, we requires that $\int_\Omega\psi\, d{\bf x} =0$.  This is the case since we use the initial data
	\begin{equation}
\psi(\, \cdot \, ,0) = \partial_t\phi(\, \cdot \, ,0) \equiv 0\qquad\mbox{in}\  \Omega  .
	\end{equation}
A simple calculation~\cite{wang11a,wang10b} shows that sufficiently regular solutions dissipate the pseudo-energy at the rate
	\begin{equation}
d_t{\mathcal E}	= -\left( \psi, \psi \right)_{H^{-1}} \le 0  .
	\label{pde-psuedo-energy-decrease}
	\end{equation}
In other words, the pseudo energy is non-increasing in time.  The primary motivation in the convex splitting framework is to design fully and semi-discrete schemes that mimic this pseudo-energy dissipation~\cite{baskaran12, wang11a, wang10b}.

The first order convex splitting scheme for (\ref{MPFC-eq}) was proposed and analyzed in a recent article~\cite{wang11a}, as we have mentioned. However, the extension to the second order convergence analysis is highly non-trivial, mainly due to an $O(s^2)$ numerical error between the centered difference of $\phi$ and the mid-point average of $\psi$. As observed in \cite{baskaran12}, the introduction of the variable $\psi$ greatly facilitates the numerical implementation.  However, if one is not careful, the above-mentioned $O (s^2)$ numerical error might seem to introduce a reduction of temporal accuracy, because of the second order time derivative involved in the equation.  To overcome this difficulty in the paper, we have to perform a higher order consistency analysis by an asymptotic expansion; as a result, the constructed approximate solution satisfies the numerical scheme with a higher order truncation error.  A projection of the exact solution onto the Fourier space is taken so that an optimal regularity requirement is obtained. 

Second order convergence analysis has always been very challenging  for nonlinear hyperbolic equation with second order temporal derivative involved.  The nonlinear error term must be carefully expanded, and a discrete Sobolev inequality is needed to bound the discrete $L^\infty$ and $W^{1,4}$ norms of the numerical error function. In addition, we need to take inner product with the error equation by the (discrete) time derivative of the numerical error, because of the hyperbolic nature of MPFC equation. In the end, a full second order convergence in a discrete $L^\infty \left(0,T; H^3\right)$ norm is established.  

In Sec.~\ref{sec-scheme} we define the second-order convex splitting scheme and restate some solvability and stability results from~\cite{baskaran12}.  In Sec.~\ref{sec-main-results} we present the convergence analysis for the second order scheme. We give some concluding remarks in Sec.~\ref{sec-conclusions}.  Moreover, some technical details of the forthcoming analysis are provided in two appendices.  In App.~\ref{app-tools} we give the finite difference background for the analysis, including our notation, some of the necessary difference operators, and the some useful inequalities. In a second appendix, App.~\ref{app-consistency}, we give the details of the consistency analyses related to our scheme.

	\section{The Second-Order Scheme and its Properties}
	\label{sec-scheme}
	
Here we redefine our second-order convex splitting scheme from~\cite{baskaran12}.  We also restate some of the unconditional solvability and stability results for this scheme.  We note that we used a different non-dimensional scaling of the MPFC equation~\eqref{MPFC-eq} in~\cite{baskaran12} than we do here, and some of the restated results below will be in a slightly modified form.  However, this difference is only superficial.  The reader is directed to App.~\ref{app-tools} for an introduction to the notation, as well as some of the standard tools from cell centered finite differences, that are used below.
	
	\subsection{Discrete Energy and the Convex-Splitting Scheme}

We first introduce a  fully discrete energy that is consistent with the continuous space energy (\ref{energy}).  In particular, define the discrete energy $F:{\mathcal C}_{\overline{m}\times\overline{n}}\rightarrow \mathbb{R}$ to be
	\begin{equation}
F(\phi) := \frac{1}{4}\nrm{\phi}_4^4+\frac{\alpha}{2} \nrm{\phi}_2^2 -\nrm{\nabh\phi}_2^2+ \frac{1}{2}\nrm{\Dh\phi}_2^2. \label{discrete-energy}
	\end{equation}
The discrete analogue to (\ref{energy-plus-time}) is 
	\begin{equation}\label{energy-plus-time-discrete}
{\mathcal F}\left(\phi,\psi\right) := F\left(\phi\right) + \frac{\beta}{2}\nrm{\psi}_{-1}^2,
	\end{equation}
defined for any $\phi\in{\mathcal C}_{\overline{m}\times\overline{n}}$ and any $\psi\in H$.  The norms above, including the  ``$-1$" norm, are defined in App.~\ref{app-tools}.

Note that if $\phi\in {\mathcal C}_{\overline{m}\times\overline{n}}$ is periodic, then it is easy to see that the energies 
	\begin{equation}
F_c(\phi) = \frac{1}{4} \nrm{\phi}_4^4 +\frac{\alpha}{2} \nrm{\phi}_2^2 
+ \frac{1}{2}\nrm{\Dh\phi}_2^2 \quad\mbox{and}\quad  F_e(\phi) =   \nrm{\nabh\phi}_2^2
	\end{equation}
are convex~\cite{wang11a,wise09}.  Hence $F$, as defined in (\ref{discrete-energy}), admits  the convex splitting $F = F_c-F_e$.  Our second-order scheme will exploit this decomposition of $F$.  Eyre~\cite{eyre98} is often credited with popularizing the idea that the numerical scheme should respect the convexity structure of the energy for the purposes of numerical stability and solvability.  His original scheme was first-order accurate in time and was restricted to non-conserved gradient flows.  But this approach has been extended to craft schemes for a number of gradient-flow equations of parabolic type; see for example~\cite{chen11, shen11, wang10a, wise09, wise10}.  The convex splitting framework was extended for the hyperbolic MPFC equation (\ref{MPFC-eq}) in~\cite{wang11a,wang10b}.  We extended the framework for second-order schemes in~\cite{baskaran12, hu09, shen11}.

The following second-order convex splitting scheme for the MPFC equation is from our recent paper~\cite{baskaran12}: given $\phi^{k-1},\, \phi^k, \, \psi^k \in {\mathcal C}_{\overline{m}\times\overline{n}}$ periodic, find $\phi^{k+1}, \, \psi^{k+1}, \, \mu^{k+1/2} \in {\mathcal C}_{\overline{m}\times\overline{n}}$ periodic such that
         \begin{eqnarray}
\beta\left(\psi^{k+1} - \psi^k\right) &=&  s\Delta_h \mu^{k+1/2} - s \psi^{k+1/2}  , 
    \label{s-2nd-1}
	\\
\mu^{k+1/2} &=&  \chi\left(\phi^{k+1},\phi^k \right) + \alpha \phi^{k+1/2}  + 2 \Delta_h \hat{\phi}^{k+1/2}   + \Delta_h^2 \phi^{k+1/2}  ,
	\label{s-2nd-2}
	\\
\phi^{k+1} - \phi^k &=&  s\psi^{k+1/2}  ,
	\label{s-2nd-3}
	\end{eqnarray}
where
	\begin{displaymath}
\phi^{k+\frac{1}{2}} := \frac{\phi^{k+1} + \phi^k }{2}, \quad \chi(\phi,\psi) :=  \frac{\phi^2+\psi^2}{2}\phi^{k+\frac{1}{2}} ,\quad \hat{\phi}^{k+\frac{1}{2}} := \frac{3 \phi^k - \phi^{k-1}}{2} .
	\end{displaymath}
It is obvious that $\chi(\phi,\phi) = \phi^3$.  In~\cite{baskaran12}, we used the initial data
	\begin{equation}
\phi^{-1} \equiv \phi^0  , \quad \psi^0 \equiv 0 .
	\end{equation}
Note that $\phi^{-1} \equiv \phi^0$ is an $O (s^2)$ approximation to the phase variable at the time ``ghost" point $k=-1$; such an initial error does not affect the order of numerical accuracy. 

By simple manipulations we obtain the following equivalent formulation~\cite{baskaran12}:
	\begin{eqnarray}
\left(1+\frac{2\beta}{s}\right) \phi^{k+1}- s \Dh\mu^{k+1/2} &=& \left(1+\frac{2\beta}{s}\right) \phi^k + 2\beta \psi^k  ,
	\label{full-disc-2nd-a-1}
	\\
\psi^{k+1} &=& \psi^k +\frac{2}{s}\left(\phi^{k+1}-\phi^k \right)  ,
	\label{full-disc-2nd-a-2}
	\end{eqnarray}
which shows that the equations may be decoupled.  In fact, we can obtain $\phi^{k+1}$ first by solving (\ref{full-disc-2nd-a-1}) and then update $\psi^{k+1}$ using (\ref{full-disc-2nd-a-2}).  Clearly the solvability of the scheme rests on the solvability of Eq.~(\ref{full-disc-2nd-a-1}).

	\subsection{Mass Conservation, Unique Solvability and Unconditional Energy Stability}

Mass conservation, unconditional unique solvability, and unconditional psuedo-energy stability were established in~\cite{baskaran12}.  We recall these facts here, though the reader is directed to the reference for details for the details.	  There are two modifications below from what is in~\cite{baskaran12}. First, our non-dimensional scaling of~\eqref{MPFC-eq} is slightly different, and, second, we use different initializations for our multistep, convex splitting  scheme.
 
	\begin{theorem}
The second order MPFC scheme (\ref{full-disc-2nd-a-1}) -- (\ref{full-disc-2nd-a-2}) is uniquely solvable for any time step-size $s>0$ and, moreover, solutions are mass-conservative, \emph{i.e}, $\ciptwo{\phi^k}{{\bf 1}} = \ciptwo{\phi^0}{{\bf 1}}$, for all $k = 1, 2, \ldots$.
	\end{theorem}

Before we state the next result, which is proved in~\cite{baskaran12}, we introduce a third fully discrete energy: for each time step $k \ge 1$, set
	\begin{equation}
\tilde{\mathcal{F}}\left(\phi^k,\phi^{k-1},\psi^k\right) := \mathcal{F}\left(\phi^k,\psi^k\right) +\frac{1}{2}\nrm{\nabh\left(\phi^k-\phi^{k-1}\right)}_2^2 \ .
	\end{equation}

	\begin{theorem}
	\label{energy-stable}
The second order MPFC scheme (\ref{full-disc-2nd-a-1}) -- (\ref{full-disc-2nd-a-2}) (or equivalently (\ref{s-2nd-1}) -- (\ref{s-2nd-3})) is unconditionally energy stable.  In particular, suppose that  $\phi^k$, $\psi^k$, $\phi^{k-1}\in \mathcal{C}_{\overline{m}\times\overline{n}}$ are periodic, and that $\phi^{k+1}$, $\mu^{k+\hf}$, $\psi^{k+1} \in \mathcal{C}_{\overline{m}\times\overline{n}}$ is a periodic solution triple to (\ref{s-2nd-1}) -- (\ref{s-2nd-3}).  Then, for any $k\ge 0$, 
	\begin{eqnarray}
\tilde{\mathcal{F}}\left(\phi^{k+1},\phi^{k},\psi^{k+1}\right) + s\nrm{\psi^{k+\hf}}_{-1}^2 + \frac{s^4}{2}\nrm{\nabh\left(D^2_s\phi^k\right)}_2^2 = \tilde{\mathcal{F}}\left(\phi^k,\phi^{k-1},\psi^k\right) ,
	\nonumber
	\\
&&
	\label{primary-stability-result}
	\end{eqnarray}
where
	\begin{equation}
D^2_s\phi^k := \frac{1}{s^2}\left(\phi^{k+1}-2\phi^k+\phi^{k-1} \right)  .
	\end{equation}
	\end{theorem}

This next result follows by summing Eq.~(\ref{primary-stability-result}) from $k = 0$ to $k=\ell-1$.

	\begin{corollary}
With the same assumptions as in Thm.~\ref{energy-stable} we have
	\begin{eqnarray}
\tilde{\mathcal{F}}\left(\phi^\ell,\phi^{\ell-1},\psi^\ell\right) +s\sum_{k = 0}^{\ell-1}\nrm{\psi^{k+\hf}}_{-1}^2 +\frac{s^4}{2}\sum_{k = 0}^{\ell-1}\nrm{\nabh\left(D^2_s\phi^k\right)}_2^2 &=& \tilde{\mathcal{F}}\left(\phi^0,\phi^{-1},\psi^0\right) 
	\nonumber
	\\
& =& F\left(\phi^0\right).
	\end{eqnarray}
	\end{corollary}

Using Lems.~\ref{sobolev-2d-final} and \ref{embedding}, we find
	\begin{lemma}
	\label{H2-stable}
Suppose that $\phi \in{\mathcal C}_{\overline{m}\times\overline{n}}$ is periodic. Then the following estimates hold:
	\begin{eqnarray}
F(\phi) &\ge& C_5 \nrm{\phi}^2_{2,2} - \frac{L_xL_y}{4}  , \label{H2-energy-estimate}
	\\
F(\phi) &\ge& C_6 \nrm{\phi}_\infty^2 - \frac{L_xL_y}{4} , \quad C_6 := \frac{C_5}{C_2}  , \label{infty-energy-estimate}
	\\
F(\phi) &\ge& C_7 \nrm{\nabh\phi}_4^2 - \frac{L_xL_y}{4} , \quad C_7 := \frac{C_5}{C_4}  , \label{W14-energy-estimate}
	\end{eqnarray}
where $C_5>0$ and only depends upon $\alpha$.
	\end{lemma}

Using the last two results and the simple estimate
	\begin{equation} 
F\left(\phi^k\right) \le \tilde{\mathcal{F}}\left(\phi^k, \phi^{k-1}, \psi^k\right)  ,
	\label{energy-stab}
	\end{equation} 
for any $0\ge 1$, we obtain
	\begin{theorem}
Let $\Phi$ be a sufficiently regular, periodic solution to \eqref{MPFC-eq} on $\Omega_T = (0,L_x)\times (0\times L_y)\times(0,T)$, with $\partial_t\Phi(\, \cdot\, ,\, \cdot\, ,0) \equiv 0$ and $\phi^0_{i,j} = \phi^{-1}_{i,j} := \Phi(p_i,p_j,0)$, $\psi^0 \equiv 0$. Suppose $E$ is the continuous energy (\ref{energy}) and $F$ is the discrete energy (\ref{discrete-energy}).  Let $\phi^k_{i,j}\in {\mathcal C}_{\overline{m}\times\overline{n}}$ be the $k^{\rm th}$ periodic solution of  (\ref{full-disc-2nd-a-1}) and (\ref{full-disc-2nd-a-2}) for $1\le k\le \ell$. Set 
	\begin{equation}
M_0 := E\big(\Phi(\, \cdot\, ,\, \cdot\, ,0)\big) + C_8 L_xL_y  ,
	\end{equation}
where $C_8>0$ is a constant that does not depend on either $s$ or $h$.  Then we have the following estimates:
	\begin{eqnarray}
s\sum_{k = 0}^{\ell-1}\nrm{\psi^{k+\hf}}_{-1}^2 + \frac{s^4}{2}\sum_{k = 0}^{\ell-1}\nrm{\nabh\left(D^2_s\phi^k\right)}_2^2  &\le& M_0  ,
	\\	
\max_{0\le k\le \ell}\nrm{\phi^k}_{2,2} &\le& \sqrt{\frac{M_0}{C_5}} \ =: C_9 , \label{h2-uniform-bound}
	\\
\max_{0\le k\le \ell}\nrm{\phi^k}_\infty &\le& \sqrt{\frac{M_0}{C_6}} \ =: C_{10}  , \label{infty-uniform-bound}
	\\
\max_{0\le k\le \ell}\nrm{\nabh\phi^k}_4 &\le& \sqrt{\frac{M_0}{C_7}} \ =: C_{11}  . \label{W14-uniform-bound}
	\end{eqnarray} 
	\end{theorem}

	\begin{theorem}
Suppose that $\Phi(x,y,t)$ is a periodic solution of the MPFC equation (\ref{dyn-conserve}), with the regularity assumed in Thm.~\ref{thm-local-truncation-error} below, such that $\partial_t\Phi(x,y,0) = 0$.  Then we have the following estimates:
	\begin{eqnarray}
\nrm{\Phi}_{L^\infty\left(0,T;H^2\left(\Omega\right)\right)} &\le& \sqrt{ C_{12}\left( E\left(\Phi\left(x,y,0\right)\right) +\frac{L_xL_y}{4}\right)} =: C_{13} \ ,
	\\
\nrm{\Phi}_{L^\infty\left(0,T;L^\infty\left(\Omega\right)\right)} &\le& \sqrt{ C_{14}\left( E\left(\Phi\left(x,y,0\right)\right) +\frac{L_xL_y}{4}\right)} =: C_{15} \ ,
	\\
\nrm{\Phi}_{L^\infty\left(0,T;W^{1,4}\left(\Omega\right)\right)} &\le&   \sqrt{ C_{16}\left( E\left(\Phi\left(x,y,0\right)\right) +\frac{L_xL_y}{4}\right)} =: C_{17} \ ,
	\end{eqnarray}
for any $T\ge 0$, where $C_{12},\, C_{14},\, C_{16} >0$ are constants that are independent of $T$.
	\end{theorem}

	\section{Error Estimate for the Second Order Scheme}
	\label{sec-main-results}

We now prove an error estimate for the second order scheme  (\ref{full-disc-2nd-a-1}) -- (\ref{full-disc-2nd-a-2}) for the MPFC equation. The following estimate, proved in~\cite{wise09}, shows control of the backward diffusion term.
	
	\begin{lemma}\label{special-lemma}
Suppose that $\phi\in {\mathcal C}_{\overline{m}\times\overline{n}}$ is periodic and that $\Dh\phi\in {\mathcal C}_{\overline{m}\times\overline{n}}$ is also periodic. Then 
	\begin{equation}
\nrm{\Dh \phi}_2^2 \le \frac{1}{3\epsilon^2}\nrm{\phi}_2^2 +\frac{2\epsilon}{3} \nrm{\nabh\left(\Dh \phi \right)}_2^2,
	\end{equation}
valid for arbitrary $\epsilon>0$.
	\end{lemma}

In addition, a control of the error related to the nonlinear term in the second order scheme is needed. 
	\begin{lemma}
	\label{control-nonlinear}
Suppose $\Phi^k,\, \Phi^{k+1} , \, \phi^k , \, \phi^{k+1} \in {\mathcal C}_{\overline{m}\times\overline{n}}$ are periodic and denote their differences by $\tilde{\phi}^k :=\Phi^k-\phi^k$ and $\tilde{\phi}^{k+1} :=\Phi^{k+1}-\phi^{k+1}$.   Then we have 
	\begin{eqnarray} 
&& \hspace{-0.2in} \nrm{ \Dh \left( \chi\left(\Phi^{k+1},\Phi^k\right) - \chi\left(\phi^{k+1},\phi^k\right) \right) }_2  \nonumber 
\\
  &\le& C_{18}  \biggl\{ K_1^2  \cdot  \left( \nrm{\Dh \tilde{\phi}^{k+1}  }_2 + \nrm{\Dh \tilde{\phi}^k }_2 \right)
  +  K_1 K_4 \left( \nrm{ \nabla_h \tilde{\phi}^{k+1} }_4 + \nrm{ \nabla_h \tilde{\phi}^k }_4  \right)  
	\nonumber
	\\
&&+ \left( K_1 K_3 + K_4^2 \right) \cdot \left( \nrm{  \tilde{\phi}^{k+1} }_\infty  +  \nrm{  \tilde{\phi}^k }_\infty \right) 
   + \left(  K_5^2  + K_1 K_2  \right)  \cdot \left( \nrm{  \tilde{\phi}^{k+1} }_2  +  \nrm{  \tilde{\phi}^k }_2 \right)   \biggr\} \ , 
	\nonumber
	\\
&& 
	\label{(6.52)}
	\end{eqnarray}
with 
	\begin{eqnarray} 
K_1 &=& \nrm{ \Phi^{k+1}}_\infty  + \nrm{ \Phi^k }_\infty  + \nrm{ \phi^{k+1} }_\infty  + \nrm{ \phi^k }_\infty  ,
	\nonumber 
	\\
K_2 &=&  \nrm{ \Dh^x \Phi^{k+1} }_\infty + \nrm{ \Dh^x \Phi^k }_\infty + \nrm{ \Dh^y \Phi^{k+1} }_\infty + \nrm{ \Dh^y \Phi^k }_\infty  ,  
	\nonumber 
	\\
K_3 &=&  \nrm{ \Dh^x \phi^{k+1} }_2 + \nrm{ \Dh^x \phi^k }_2 + \nrm{ \Dh^y \phi^{k+1} }_2 + \nrm{ \Dh^y \phi^k }_2  ,
	\nonumber    
	\\
K_4 &=& \nrm{ \nabla_h \Phi^{k+1} }_4  + \nrm{ \nabla_h \Phi^k }_4  + \nrm{ \nabla_h \phi^{k+1} }_4  + \nrm{ \nabla_h \phi^k }_4  , 
	\nonumber
	\\
K_5 &=& \nrm{ \nabla_h \Phi^{k+1} }_\infty  + \nrm{ \nabla_h \Phi^k }_\infty  , 
	\end{eqnarray}
and $C_{18}$ is a positive constant that is independent of $h$.
	\end{lemma}
	
	\begin{proof} 
First, careful expansions yeild the following nonlinear error decompositions: 
	\begin{eqnarray}  
\left(\Phi^{k+1}\right)^3 - \left(\phi^{k+1}\right)^3 &=&  \left(  \left(\Phi^{k+1}\right)^2 + \Phi^{k+1} \phi^{k+1} + \left(\phi^{k+1}\right)^2 \right) \tilde{\phi}^{k+1} \ ,
	\label{nonlinear-1-1}
	\\
\left(\Phi^{k+1}\right)^2 \Phi^k - \left(\phi^{k+1}\right)^2 \phi^k &=&  \left(  \Phi^{k+1} + \phi^{k+1} \right) \Phi^k \tilde{\phi}^{k+1} + \left(\phi^{k+1}\right)^2 \tilde{\phi}^k \ ,
	\label{nonlinear-1-2}  
	\\
\Phi^{k+1} \left(\Phi^k\right)^2 - \phi^{k+1} \left(\phi^k\right)^2 &=&  \left(  \Phi^k + \phi^k \right) \Phi^{k+1} \tilde{\phi}^k  + \left(\phi^k\right)^2 \tilde{\phi}^{k+1} \ ,
	\label{nonlinear-1-3}    
	\\
\left(\Phi^k\right)^3 - \left(\phi^k\right)^3 &=&  \left(  \left(\Phi^k\right)^2 + \Phi^k \phi^k + \left(\phi^k\right)^2 \right) \tilde{\phi}^k \ .
	\label{nonlinear-1-4}
	\end{eqnarray}
Meanwhile, a detailed calculation yields the following finite difference expansion:
	\begin{eqnarray}
\Dh^x \left( f g h \right)_{i,j} &=&  f_{i,j} g_{i,j} ( \Delta_h^x) h_{i,j} + f_{i,j} h_{i,j} ( \Delta_h^x) g_{i,j}  +  g_{i,j} h_{i,j} ( \Delta_h^x) f_{i,j}
	\nonumber 
	\\
&&  + f_{i,j}  \left(  D_x g_{i+\hf,j} D_x h_{i+\hf,j}  + D_x g_{i-\hf,j} D_x h_{i-\hf,j}   \right)
	\nonumber 
	\\
&&  + g_{i,j}  \left(  D_x f_{i+\hf,j} D_x h_{i+\hf,j}  + D_x f_{i-\hf,j} D_x h_{i-\hf,j}   \right)
	\nonumber   
	\\
&& + h_{i+1,j}  D_x f_{i+\hf,j} D_x g_{i+\hf,j}  + h_{i-1,j} D_x f_{i-\hf,j} D_x g_{i-\hf,j} \  . 
	\label{(6.50)}
	\end{eqnarray}
An analogous formula for $\Dh^y\left( f g h \right)_{i,j}$ holds by symmetry.   First, we bound all of the terms in the expansion of $\Dh^x \left( (\Phi)^3 - (\phi)^3 \right)$.  For brevity, we only show how this is done for one term, namely, $\Delta_h^x\left(\left(\phi^k\right)^2  \tilde{\phi}^k\right)$.  The expansion is given by 
	\begin{equation} 
\Delta_h^x  \left( \left(\phi^k\right)^2  \tilde{\phi}^k \right)_{i,j}  = N^{(1)}_{i,j} + 2 N^{(2)}_{i,j} + 2 N^{(3)}_{i,j} +  N^{(4)}_{i,j}  ,   
	\end{equation}
with
	\begin{eqnarray}
N^{(1)}_{i,j} &=& \left( \phi^k_{i,j} \right)^2  \Delta_h^x \tilde{\phi}^k_{i,j} ,  \quad 
N^{(2)}_{i,j}  =  \phi^k_{i,j} \tilde{\phi}^k_{i,j}  \Delta_h^x \phi^k_{i,j} \ ,
	\\
N^{(3)}_{i,j} &=&  \phi^k_{i,j}  \left(  D_x \phi^k_{i+1/2,j} D_x \tilde{\phi}^k_{i+1/2,j} + D_x \phi^k_{i-1/2,j} D_x \tilde{\phi}^k_{i-1/2,j}   \right) \ ,
	\\
N^{(4)}_{i,j} &=&  \tilde{\phi}^k_{i+1,j}  \left( D_x \phi^k_{i+1/2,j} \right)^2 + \tilde{\phi}^k_{i-1,j} \left( D_x \phi^k_{i-1/2,j} \right)^2 \ .
	\label{nonlinear-2}
	\end{eqnarray}
Discrete H\"older's inequalities can be applied to bound all of the above terms as follows:
	\begin{eqnarray}
  && 
  \left\| N^{(1)} \right\|_2 
  \le  \left\| \phi^k  \right\|_{\infty}^2 \cdot \left\|   \Delta_h^x \tilde{\phi}^k  \right\|_2 
  \le  \left\| \phi^k  \right\|_{\infty}^2 \cdot \left\|   \Delta_h \tilde{\phi}^k  \right\|_2 ,  
  \label{nonlinear-3-1}
\\
  && 
  \left\| N^{(2)} \right\|_2   
  \le  \left\| \phi^k  \right\|_{\infty} \cdot \left\|   \Delta_h \phi^k  \right\|_2 
  \cdot \left\|   \tilde{\phi}^k  \right\|_\infty  ,   \label{nonlinear-3-2}  
\\
  && 
  \left\| N^{(3)} \right\|_2   
  \le  2 \left\| \phi^k  \right\|_{\infty} \cdot \left\|   \nabla_h \phi^k  \right\|_4 
  \cdot \left\|   \nabla_h \tilde{\phi}^k  \right\|_4  ,   \label{nonlinear-3-3}   
\\
  && 
  \left\| N^{(4)} \right\|_2   
  \le  2 \left\| \nabla_h \phi^k  \right\|_4^2 
  \cdot \left\|   \tilde{\phi}^k  \right\|_\infty  ,   \label{nonlinear-3-4}   
\end{eqnarray} 
with repeated application of Lem.~\ref{laplacian-estimate}.  The nonlinear error term $\Dh^y \left( (\Phi)^3 - (\phi)^3 \right)$ can be analyzed in exactly the same way. Combining the estimates using the triangle inequality gives the result (\ref{(6.52)}) and  the lemma is proven.
	\end{proof} 

We now establish an error estimate for the fully discrete second order convex splitting scheme for the MPFC equation.  We do this in three steps.  First, we derive a local truncation error for a finite Fourier projection of the exact solution to the MPFC equation (\ref{MPFC-eq}).  Second, we derive an estimate of the difference between our numerical solution to the scheme (\ref{s-2nd-1}) -- (\ref{s-2nd-3}) and this finite Fourier projection.  Third, we use the triangle inequality to derive our global error estimate.

In the rest of the paper, for notational simplicity only, we will assume $L_x=L_y=L$, and hence $\Omega = (0,L)^2$.  As a consequence we have $m=n=N$, where we may assume $N$ is even.  The more general rectangular case can be handled straightforwardly.  Now, suppose that $\Phi$ has the following Fourier series representation on $\Omega$:
	\begin{equation} 
\Phi (x,y,t) =   \sum_{k,l=-\infty}^{\infty}  \widehat{\Phi}_{k,l} (t) {\rm e}^ {\frac{2 \pi {\rm i}}{L} \left(k\, x +l\, y\right)} \ ,
	\end{equation} 
with
	\begin{equation} 
\widehat{\Phi}_{k,l} (t) =  \frac{1}{| \Omega |} \int_{\Omega}  \Phi (x,y,t) {\rm e}^{-\frac{2 \pi {\rm i}}{L} \left(k\, x +l\, y\right)} \,  dx \, dy \ .   
	\label{consistency-Phi-1}
	\end{equation}
The (finite Fourier) projection of $\Phi$ onto the space ${\cal B}^{N/2}$, consisting of all trigonometric polynomials in $x$ and $y$ of degree up to $N/2$, is defined as 
	\begin{equation} 
\Phi_N (x,y,t) := {\cal P}_N \Phi(x,y,t)  := \sum_{k,l=-N/2+1}^{N/2} \widehat{\Phi}_{k,l} (t) {\rm e}^{\frac{2 \pi {\rm i}}{L} \left(k\, x +l\, y\right)} \  .  
	\label{consistency-Phi-2}
	\end{equation}  
Define
	\begin{equation} 
\Psi_N (x,y,t) := \partial_t \Phi_N (x,y,t) - \frac{s^2}{12} \partial_t^3 \Phi_N (x,y,t) \ .  
	\label{consistency-Phi-3}
	\end{equation}
	
For any function $G=G(x,y,t)$, given $s>0$ and $k>0$, we define $G^k(x,y) := G(x,y,s\cdot k)$.

	\begin{theorem}
	\label{thm-local-truncation-error}
Suppose the unique periodic solution for the MPFC equation (\ref{dyn-conserve}) is given by 
	\begin{eqnarray} 
\Phi \in H^4 \left(0,T; L^2 \left( \Omega \right)\right)  &\cap& L^\infty \left(0,T; H^{8} \left(\Omega\right)\right)
	\nonumber
	\\
&\cap& W^{2,\infty} \left(0,T; H^{2} \left(\Omega\right)\right) \cap H^{2} \left(0,T; H^{6} \left(\Omega\right)\right) \ ,
	\label{MPFC-regularity}
	\end{eqnarray} 
for $T< \infty$. Set $\Psi := \partial_t\Phi$.  Then 
	\begin{eqnarray}
\beta \, \frac{\Psi_N^{k+1} - \Psi_N^k}{s} &=& \Dh \left( \chi\left( \Phi_N^{k+1},\Phi_N^k \right) + \frac{\alpha}{2}\left(\Phi_N^{k+1} + \Phi_N^k\right) + \Delta_h \left( 3 \Phi_N^k -  \Phi_N^{k-1} \right)  \right)  
	\nonumber 
	\\
&& + \frac{1}{2}\Delta_h^3 \left(\Phi_N^{k+1} + \Phi_N^k\right)  -  \frac{\Phi_N^{k+1} - \Phi_N^k}{s}  +  \tau_1^k \ ,
	\label{truncation-equation-1}
	\\
\frac{\Phi_N^{k+1} - \Phi_N^k}{s} &=&  \frac{1}{2}\left(\Psi_N^{k+1} + \Psi_N^k\right) + s \tau_2^k \ ,
	\label{truncation-equation-2}
	\end{eqnarray}
where $\tau_1^k$ and $\tau_2^k$ satisfy
	\begin{equation}
\nrm{\tau_i}_{L_s^2\left(0,T; L_h^2(\Omega)\right)} := \sqrt{s \sum_{k=0}^{T/s} \nrm{\tau_i^{k+1}}^2_2\ } \le M \left(s^2 + h^2 \right) \ ,
	\label{discrete-norm-truncation}  
	\end{equation} 
for $i=1,2$, with
	\begin{eqnarray}  
M  &\le& C \Bigl(\nrm{\Phi}_{H^4 \left(0,T; L^2 \left( \Omega \right)\right)} 
+ \nrm{\Phi}_{W^{2,\infty} \left(0,T; H^2 \left( \Omega \right)\right)} 
+ \nrm{\Phi}_{H^{2} \left(0,T; H^{6}\left(\Omega\right)\right)}	
	\nonumber 
\\
  && 
    + \nrm{\Phi}_{L^\infty \left(0,T; H^{4}\left(\Omega\right)\right)}^{3}   \cdot 
  \left( 1 + \nrm{\Phi}_{H^2 \left(0,T; H^{2}\left(\Omega\right)\right)}^{2}  \right) 	
+ \nrm{\Phi}_{L^\infty \left(0,T; H^{8} \left(\Omega\right)\right)} \Bigr) \ .
	\label{truncation-error} 
	\end{eqnarray}
	\end{theorem}

The details of the proof are technical and are contained in the appendix.

\begin{rem} 
The constructed solution (\ref{consistency-Phi-2}) comes from the Fourier projection 
of the exact solution $\Phi$. 
The reason for the choice of $\Phi_N$ instead of $\Phi$ is the fact that 
$\Phi_N \in {\cal B}^{N/2}$, which in turn gives a local truncation error estimate 
without involving with an aliasing error, as can be seen in the appendix. 
Meanwhile, an $O (s^2)$ correction term is added in the construction 
(\ref{consistency-Phi-3}) for $\Psi_N$ so that a higher order consistency is 
obtained in (\ref{truncation-equation-2}). Such a correction term is based on an 
asymptotic expansion of the numerical scheme and the resulting higher order 
consistency is crucial in the stability and convergence analysis. 
Finally, a numerical convergence of the numerical solution to $(\Phi_N,\Psi_N)$ 
is equivalent to its convergence to the exact solution $(\Phi, \Psi)$, since 
$\Phi$ is a spectrally accurate approximation, and $\Psi_N$ is an $O (s^2)$ 
approximation to $\Psi$.  
\end{rem}

	\begin{theorem}
	\label{thm-projection-error-estimate}
Suppose $\Phi$, $\Phi_N$, $\Psi$ and $\Psi_N$ are as in the last theorem.  Define $\tilde{\phi}_{i,j}^k := \Phi_N^k\left(h\cdot i,h\cdot j\right)-\phi_{i,j}^k$ and $\tilde{\psi}_{i,j}^k := \Psi_N^k\left(h\cdot i,h\cdot j\right)-\psi_{i,j}^k$, where $\phi^k_{i,j},\, \psi^k_{i,j} \in {\mathcal C}_{\overline{m}\times\overline{n}}$ are the $k^{\rm th}$ periodic solutions of (\ref{s-2nd-1}) -- (\ref{s-2nd-3}),  or equivalently,  (\ref{full-disc-2nd-a-1}) -- (\ref{full-disc-2nd-a-2}), with $\phi^0_{i,j} := \Phi^0_{i,j}$,  $\phi^{-1}_{i,j}=\phi^0_{i,j}$ and $\psi^0_{i,j}=0$.  Then
	\begin{equation}\label{(6.39)}
\nrm{\tilde{\phi}^k}_2 +  \nrm{\nabh \left(\Dh \tilde{\phi}^k\right)}_2  
\le C\left( s^2 + h^2 \right) \ ,
	\end{equation} 
provided $s$ is sufficiently small, for some $C>0$ that is independent of $h$ and $s$.
	\end{theorem}
	\begin{proof}
Subtracting (\ref{s-2nd-1}) -- (\ref{s-2nd-3}) from  (\ref{truncation-equation-1}), (\ref{truncation-equation-2}) yields
	\begin{eqnarray}
\beta \, \frac{\tilde{\psi}^{k+1} - \tilde{\psi}^k}{s} &=&  \Dh \biggl( \chi\left(\Phi_N^{k+1},\Phi_N^k\right) - \chi\left(\phi^{k+1},\phi^k\right) + \alpha \tilde{\phi}^{k+1/2}    
	\nonumber 
	\\
&&  +  \Dh \left( 3 \tilde{\phi}^k - \tilde{\phi}^{k-1} \right)  + \Delta_h^2 \tilde{\phi}^{k+1/2}  \biggr) -  \frac{\tilde{\phi}^{k+1} - \tilde{\phi}^k}{s}  +  \tau_1^k \ ,
	\label{error-equation-1}
	\\
\frac{\tilde{\phi}^{k+1} - \tilde{\phi}^k}{s} &=&  \tilde{\psi}^{k+1/2} + s \tau_2^k \ ,
	\label{error-equation-2}
	\end{eqnarray}
where $\tilde{\phi}^{k+1/2} := \frac12 \left( \tilde{\phi}^{k+1} + \tilde{\phi}^k\right)$ and  $\tilde{\psi}^{k+1/2} := \frac12 \left( \tilde{\psi}^{k+1} + \tilde{\psi}^k\right)$.  Taking the inner product with the error difference function $h^2\left(\tilde{\phi}^{k+1} - \tilde{\phi}^k\right)$ gives 
	\begin{eqnarray}
&& \hspace{-0.7in} h^2\, \ciptwo{\tilde{\phi}^{k+1} - \tilde{\phi}^k}{\tau_1^k} + h^2\, \ciptwo{\tilde{\phi}^{k+1} - \tilde{\phi}^k}{ \Dh \left( \chi\left(\Phi_N^{k+1},\Phi_N^k\right) - \chi\left(\phi^{k+1},\phi^k\right) \right)}
 	\nonumber 
	\\
 &=& \frac{\beta h^2}{s} \ciptwo{\tilde{\psi}^{k+1} - \tilde{\psi}^k } {\tilde{\phi}^{k+1} - \tilde{\phi}^k}  + \frac{h^2}{s} \ciptwo{\tilde{\phi}^{k+1} - \tilde{\phi}^k }{\tilde{\phi}^{k+1} - \tilde{\phi}^k}
 	\nonumber 
	\\
&&-  \alpha h^2\, \ciptwo{\tilde{\phi}^{k+1} - \tilde{\phi}^k}{\Dh \tilde{\phi}^{k+1/2}} -   h^2\, \ciptwo{\tilde{\phi}^{k+1} - \tilde{\phi}^k}{\Dh^3  \tilde{\phi}^{k+1/2}}     
	\nonumber 
	\\
& & - h^2\, \ciptwo{\tilde{\phi}^{k+1} - \tilde{\phi}^k} {\Dh^2  \left( 3 \tilde{\phi}^k - \tilde{\phi}^{k-1} \right)} \ . 
	\label{(6.42)}
	\end{eqnarray}
The first term on the right-hand-side of (\ref{(6.42)}) can be rewritten and estimated as follows.  With the help of (\ref{error-equation-2}) and an application of Cauchy's inequality we have
	\begin{eqnarray} 
  &&
\frac{h^2}{s} \ciptwo{\tilde{\phi}^{k+1} - \tilde{\phi}^k}{\tilde{\psi}^{k+1} - \tilde{\psi}^k }  = h^2\, \ciptwo{ \frac{\tilde{\psi}^{k+1} + \tilde{\psi}^k }{2} + s \tau_2^k }
{ \tilde{\psi}^{k+1} - \tilde{\psi}^k }	
	\nonumber
	\\
&=&  \frac12 \left(\nrm{\tilde{\psi}^{k+1}}_2^2 - \nrm{\tilde{\psi}^k}_2^2  \right) +  s h^2 \,  \ciptwo{\tau_2^k}{ \tilde{\psi}^{k+1} - \tilde{\psi}^k }
	\nonumber
	\\
&\ge&  \frac12 \left(\nrm{\tilde{\psi}^{k+1}}_2^2 - \nrm{\tilde{\psi}^k}_2^2  \right) - \frac12 s \nrm{\tau_2^k}_2^2 
 - s \left( \nrm{\tilde{\psi}^{k+1}}_2^2  + \nrm{\tilde{\psi}^k}_2^2 \right) \ .  
	\label{(6.43)} 
	\end{eqnarray} 
The second term on the right-hand-side of (\ref{(6.42)}) is obviously non-negative: 
	\begin{equation}
\frac{h^2}{s} \ciptwo{\tilde{\phi}^{k+1} - \tilde{\phi}^k}{\tilde{\phi}^{k+1} - \tilde{\phi}^k } = s \nrm{ \frac{\tilde{\phi}^{k+1} - \tilde{\phi}^k}{s} }_2^2 \ge 0 \ .
	\label{(6.44)}
	\end{equation}
The first term on the left-hand-side of (\ref{(6.42)}) can be controlled using Cauchy's inequality:  
	\begin{eqnarray}
h^2\, \ciptwo{\tilde{\phi}^{k+1} - \tilde{\phi}^k}{\tau_1^k}  &=&  s h^2\, \ciptwo{ \frac{\tilde{\psi}^{k+1} + \tilde{\psi}^k }{2} + s \tau^k_2}{\tau_1^k}  
	\nonumber 
	\\
&\le&  \frac{s}{4} \left( \nrm{\tilde{\psi}^{k+1} }_2^2 + \nrm{\tilde{\psi}^k }_2^2  \right)  +  s \left( \nrm{\tau_1^k}_2^2 + \frac{s^2}{2} \nrm{\tau_2^k}_2^2  \right) \ . 
	\label{(6.45)}
	\end{eqnarray}
The analysis of the convex diffusion terms can be carried out with the help of the discrete Green's identities~(\ref{green1stthm-2d}) and (\ref{green2ndthm-2d}):
	\begin{equation} 
-h^2\, \ciptwo{\tilde{\phi}^{k+1} - \tilde{\phi}^k}{\Dh \tilde{\phi}^{k+1/2}} = \frac12 \left( \nrm{\nabla_h \tilde{\phi}^{k+1} }_2^2  - \nrm{\nabla_h \tilde{\phi}^k }_2^2  \right)
	\label{(6.46)}  
	\end{equation} 
and 
	\begin{equation}  
-h^2\, \ciptwo{\tilde{\phi}^{k+1} - \tilde{\phi}^k}{\Dh^3 \tilde{\phi}^{k+1/2} } 	=  \frac12 \left( \nrm{\nabla_h\left( \Dh \tilde{\phi}^{k+1}\right) }_2^2 -  \nrm{\nabla_h\left( \Dh \tilde{\phi}^k\right)  }_2^2  \right) \ .
	\label{(6.47)}   
	\end{equation}
The concave diffusion term can be handled with the identity 
	\begin{eqnarray} 
- h^2\, \ciptwo{\tilde{\phi}^{k+1} - \tilde{\phi}^k }{\Dh^2 \left( 3 \tilde{\phi}^k - \tilde{\phi}^{k-1} \right) }  &=&  -  \nrm{\Dh \tilde{\phi}^{k+1} }_2^2 +  \frac12  \nrm{\Dh \left( \tilde{\phi}^{k+1} - \tilde{\phi}^k  \right) }_2^2 
	\nonumber 
	\\
&&  + \nrm{\Dh \tilde{\phi}^k }_2^2  -  \frac12\nrm{\Dh \left( \tilde{\phi}^k - \tilde{\phi}^{k-1}  \right) }_2^2 
	\nonumber
	\\
&&+\frac{1}{2}\nrm{\Dh\left(\tilde{\phi}^{k+1}-2\tilde{\phi}^k+\tilde{\phi}^{k-1} \right)}_2^2 \ .
	\label{(6.48)}
	\end{eqnarray} 
For the nonlinear term, we start with an application of Cauchy's inequality: 
	\begin{eqnarray}
&& \hspace{-0.5in} h^2\,\ciptwo{\tilde{\phi}^{k+1} - \tilde{\phi}^k }{\Dh \left\{ \chi\left(\Phi_N^{k+1},\Phi_N^k \right) - \chi\left(\phi^{k+1},\phi^k\right) \right\} }
	\nonumber  
	\\
&=& s h^2\,\ciptwo{ \frac{\tilde{\psi}^{k+1} + \tilde{\psi}^k }{2} + s \tau^k_2 }{\Dh \left\{ \chi\left(\Phi_N^{k+1},\Phi_N^k \right) - \chi\left(\phi^{k+1},\phi^k\right) \right\} }
	\nonumber
	\\
&\le& \frac{s}{2}  \left( \nrm{\tilde{\psi}^{k+1} }^2_2  + \nrm{\tilde{\psi}^k }^2_2 + 2 s^2 \left\| \tau^k_2 \right\|_2^2 \right)   	\nonumber 
	\\
&&  + \frac{s}{2}\nrm{ \Dh\left\{ \chi\left(\Phi_N^{k+1},\Phi_N^k \right) - \chi\left(\phi^{k+1},\phi^k\right) \right\} }^2_2 \ .
	\label{(6.49)} 
	\end{eqnarray} 
Lem.~\ref{control-nonlinear} can be used to bound the last term appearing above.  In more details, the following uniform (in time) estimates are recalled from our previous lemmas:
	\begin{eqnarray} 
\nrm{ \Phi_N^l }_\infty &\le&  \nrm{ \Phi_N^l }_{L^\infty} \le C_{15} \ ,  \quad \nrm{ \nabla_h \Phi_N^l }_4 \le C \nrm{ \nabla \Phi_N^l}_{L^\infty} +C \le C \ ,  
	\label{(6.49.b)}
	\\
\nrm{ \Delta_h^x \phi^l }_2 &\le&  C \nrm{ \phi^l }_{2,2} \le C C_{9} \ ,  \quad  \nrm{ \Delta_h^y \phi^l }_2 \le  C \nrm{ \phi^l }_{2,2} \le C_{9} \ , 
	\\
\nrm{ \phi^l }_\infty &\le& C_{10} \ ,  \quad \nrm{ \nabla_h \phi^l }_4  \le C_{11} \ ,
	\end{eqnarray}
and the following estimates are valid on the finite time interval $[0,T]$:
	\begin{eqnarray}
\nrm{ \nabla_h \Phi^{l} }_\infty &\le& \nrm{ \nabla \Phi^{l} }_{L^\infty} + C \le C \ ,
	\\
\nrm{\Dh^x\Phi^{l}}_\infty+\nrm{\Dh^y\Phi^{l}}_\infty 
&\le& \nrm{\partial_{xx}\Phi^{l}}_{L^\infty}+\nrm{\partial_{yy}\Phi^{l}}_{L^\infty}  + C 
\le C \ , 
	\label{(6.56.b)}
	\end{eqnarray}
for $l=k, k+1$, where $C$ denotes a generic positive constant that is independent of $h$. 
Applying Lem.~\ref{embedding}, Lem.~\ref{sobolev-2d-final}, and substituting estimates (\ref{(6.49.b)}) -- (\ref{(6.56.b)}) yield 
	\begin{eqnarray}
&& \hspace{-0.75in} \nrm{ \Dh \left\{ \chi\left(\Phi_N^{k+1},\Phi_N^k \right) - \chi\left(\phi^{k+1},\phi^k\right) \right\} }_2  \nonumber 
	\\
&\le&  C_{19} \left( \nrm{ \tilde{\phi}^{k+1} }_2 + \nrm{ \tilde{\phi}^k }_2 + \nrm{ \Dh \tilde{\phi}^{k+1} }_2  + \nrm{ \Dh \tilde{\phi}^k }_2 \right) \ ,  \label{(6.57)}
	\end{eqnarray}
where $C_{19}>0$ is independent of $h$ and $s$, but is dependent upon $T$ and also the exact solution $\Phi$.  Going back to (\ref{(6.49)}) and using the last estimate and Lem.~\ref{special-lemma} (with $\epsilon =1$) we obtain an estimate for the nonlinear term: 
	\begin{eqnarray} 
&&\hspace{-0.7in}  h^2\,\ciptwo{\tilde{\phi}^{k+1} - \tilde{\phi}^k }{\Dh 
\left\{ \chi\left(\Phi_N^{k+1},\Phi_N^k \right) - \chi\left(\phi^{k+1},\phi^k\right) \right\} }
	\nonumber
	\\
&\le&  \frac{s}{2}  \left( \nrm{\tilde{\psi}^{k+1} }^2_2 + \nrm{\tilde{\psi}^k }^2_2 + 2 s^2 \left\| \tau_2 \right\|_2^2 \right)  	\nonumber 
	\\
&&+ 2 s \,  C_{19}^2\left( \nrm{ \tilde{\phi}^{k+1} }^2_2 + \nrm{ \tilde{\phi}^k }^2_2  + \nrm{ \Dh \tilde{\phi}^{k+1} }^2_2  + \nrm{ \Dh \tilde{\phi}^k }^2_2 \right)
	\nonumber 
	\\
&\le&   \frac{s}{2}  \left( \nrm{\tilde{\psi}^{k+1} }^2_2  
  + \nrm{\tilde{\psi}^k }^2_2 + 2 s^2 \left\| \tau_2 \right\|_2^2 \right)  	\nonumber 
	\\
&& +  \frac{8s \,  C_{19}^2}{3}\left( \nrm{ \tilde{\phi}^{k+1} }_2^2 +  \nrm{ \tilde{\phi}^k }_2^2 + \nrm{ \nabla_h \Dh \tilde{\phi}^{k+1} }_2^2 + \nrm{ \nabla_h \Dh \tilde{\phi}^k  }_2^2     \right) \ . 
	\label{(6.58)} 
	\end{eqnarray}

Define a modified energy for the error function via
	\begin{eqnarray} 
F_1 \left( \tilde{\phi}^k \right)  &:=& \frac{\beta}{2} \nrm{\tilde{\psi}^k }_2^2 + \frac{\alpha}{2} \nrm{ \nabla_h \tilde{\phi}^k }_2^2 
+ \frac{1}{2}\nrm{\nabla_h\left( \Dh \tilde{\phi}^k\right)  }_2^2 
-  \nrm{ \Dh \tilde{\phi}^k }_2^2   \nonumber 
\\
  &&
+ \frac12 \nrm{\Dh \left( \tilde{\phi}^k - \tilde{\phi}^{k-1}  \right) }_2^2 \ .  \label{(6.60)}
	\end{eqnarray} 
A combination of (\ref{(6.42)}), (\ref{(6.44)}) -- (\ref{(6.48)}) and (\ref{(6.58)}) results in
	\begin{eqnarray} 
F_1 \left( \tilde{\phi}^{k+1} \right)  -  F_1 \left( \tilde{\phi}^k \right)   &\le&   
s\,  C_{20} \Biggl( \nrm{ \tilde{\phi}^{k+1} }^2 + \nrm{ \tilde{\phi}^k }^2 
 +  \nrm{\tilde{\psi}^{k+1} }_2^2  +  \nrm{\tilde{\psi}^k }_2^2  
   \nonumber
	\\
&&  +  \nrm{ \nabh\left( \Dh \tilde{\phi}^{k+1}\right) }_2^2  
  +  \nrm{ \nabh\left( \Dh \tilde{\phi}^k \right) }_2^2  \Biggr)  \nonumber 
\\
  &&
  +  C s \left( \nrm{\tau_1^k}_2^2 + \nrm{\tau_2^k}_2^2  \right) \ , 
	\label{(6.61)} 
	\end{eqnarray} 
where $C_{20}>0$ is independent of $h$ and $s$.
Summing over $k$ and using the fact that $F_1\left(\tilde{\phi}^1\right) \le C h^4$  yields 
	\begin{eqnarray} 
F_1\left( \tilde{\phi}^\ell \right) &\le&   2 s\, C_{20} \sum_{k=1}^\ell 
\left( \nrm{ \tilde{\phi}^k }_2^2 + \nrm{ \nabh\left( \Dh \tilde{\phi}^k\right)  }_2^2 +  \nrm{ \tilde{\psi}^k }_2^2  \right)   \nonumber 
\\
  &&
+  C \left( \nrm{\tau_1}_{L^2_s \left(0,T;L_h^2(\Omega)\right)}^2  
+  \nrm{\tau_2}_{L^2_s \left(0,T;L_h^2(\Omega)\right)}^2 \right) \nonumber
	\\
&\le&   2 s\, C_{20} \sum_{k=1}^\ell \left( \nrm{ \tilde{\phi}^k }_2^2 + \nrm{ \nabh\left( \Dh \tilde{\phi}^k\right)  }_2^2 +  \nrm{ \tilde{\psi}^k }_2^2  \right)   \nonumber 
\\
  && 
+  C M^2 T (s^2 + h^2)^2  .  \label{(6.62)}
	\end{eqnarray} 
To carry out further analysis, we introduce the positive part $F_1$: 
	\begin{eqnarray} 
F_2\left( \tilde{\phi}^k \right) &:=& \frac{\beta}{2} \nrm{\tilde{\psi}^k }^2  + \frac{\alpha}{2} \nrm{ \nabla_h \tilde{\phi}^k }^2 
+ \frac{1}{2}\nrm{ \nabh\left( \Dh \tilde{\phi}^k\right) }_2^2   \nonumber 
\\
  && 
+ \frac12 \nrm{\Dh \left( \tilde{\phi}^k - \tilde{\phi}^{k-1}  \right) }_2^2  
    =  F_1\left( \tilde{\phi}^k \right)  +  \nrm{ \Dh \tilde{\phi}^k }_2^2  \ ,  
	\label{(6.63)}
	\end{eqnarray} 
so that (\ref{(6.62)}) becomes  
	\begin{eqnarray} 
F_2\left( \tilde{\phi}^\ell \right)  &\le&  2 s\, C_{20} \sum_{k=1}^\ell \left( \nrm{ \tilde{\phi}^k }_2^2 + \nrm{ \nabh \left(\Dh \tilde{\phi}^k\right) }_2^2 +  \nrm{ \tilde{\psi}^k }_2^2  \right)  \nonumber  
	\\
&&+   \nrm{ \Dh \tilde{\phi}^\ell }^2 +  C M^2 T (s^2 + h^2)^2 \ . 
	\label{(6.64)}
	\end{eqnarray} 
To estimate the additional term $\nrm{\Dh \tilde{\phi}^\ell }_2^2$, we need a bound of 
$\nrm{\tilde{\phi}^\ell }_2$	 in terms of $\tilde{\psi}^k$. The following identity is observed: 
\begin{eqnarray} 
  \tilde{\phi}^\ell = \tilde{\phi}^0 
  + s \sum_{k=1}^\ell \frac{\tilde{\phi}^k - \tilde{\phi}^{k-1}}{s}  
  = \tilde{\phi}^0 + s \sum_{k=1}^\ell   
  \left( \frac{\tilde{\psi}^{k} + \tilde{\psi}^{k-1} }{2} + s \tau_2^k  \right) ,  
  \label{(6.65.a)}
\end{eqnarray} 
with error equation (\ref{error-equation-2}) used in the last step. 
In turn, an application of Cauchy inequality shows that 
\begin{eqnarray} 
  \nrm{\tilde{\phi}^\ell }_2^2  
  &\le& 
  2 \left\| \tilde{\phi}^0  \right\|_2^2 
  + 4 s T \sum_{k=1}^\ell \left\| \tilde{\psi}^k  \right\|_2^2 
  + 4 s^3 T \nrm{\tau_2}_{L^2_s \left(0,T;L_h^2(\Omega)\right)}^2  \nonumber 
\\
  &\le& 
  4 s T \sum_{k=1}^\ell \left\| \tilde{\psi}^k  \right\|_2^2 
  + C \left( h^4 + s^2 (s^4 + h^4 ) T \right) ,  \label{(6.65.b)}
\end{eqnarray}
in which the fact that $\left\| \tilde{\phi}^0  \right\|_2 \le C h^2$ (which comes from the 
construction (\ref{consistency-Phi-2}) of the approximate solution and the initial 
numerical data $\phi^0_{i,j} = \Phi^0_{i,j}$), along with the truncation error analysis 
(\ref{discrete-norm-truncation}).   
Therefore, using Lem.~\ref{special-lemma} shows that  
	\begin{eqnarray} 
 \nrm{\Dh \tilde{\phi}^\ell }_2^2  &\le&  \frac{1}{3 \epsilon^2} \nrm{\tilde{\phi}^\ell }_2^2 + \frac{2 \epsilon}{3}  \nrm{ \nabh\left( \Dh \tilde{\phi}^\ell\right) }_2^2  \nonumber 
	\\  
&\le&  \frac{2 s\, T}{3\epsilon^2}  \sum_{k=1}^\ell \nrm{ \tilde{\psi}^k }_2^2   +   \frac{2 \epsilon}{3}  \nrm{ \nabh\left( \Dh \tilde{\phi}^\ell\right) }_2^2  + C (s^4 + h^4 ) \ ,\label{(6.65)}
	\end{eqnarray}
for any $\epsilon  >  0$, with a trivial requirement that $s^2 T \le 1$.  Taking $ \epsilon = \frac{3}{8}$, the substitution of the estimate (\ref{(6.65)}) 
into (\ref{(6.64)}) shows that 
	\begin{eqnarray} 
F_2\left( \tilde{\phi}^\ell \right) - \frac{1}{4}\nrm{ \nabh\left( \Dh \tilde{\phi}^\ell\right) }_2^2  &\le&  s\, C_{21} \sum_{k=1}^\ell \left( \nrm{ \tilde{\phi}^k }_2^2+ \nrm{ \nabh \left(\Dh \tilde{\phi}^k \right) }_2^2 +  \nrm{ \tilde{\psi}^k }_2^2  \right) \nonumber
	\\
&&+  C M^2 T (s^2 + h^2)^2 \ , \label{(6.67)}
	\end{eqnarray} 
where $C_{21}>0$ is independent of $h$ and $s$. Introducing the more refined energy 
	\begin{eqnarray} 
F_3\left( \tilde{\phi}^k \right)  &:=& \frac{\beta}{2} \nrm{ \tilde{\psi}^k }_2^2 + \frac{\alpha}{2} \nrm{ \nabh \tilde{\phi}^k }_2^2 +  \frac{1}{4} \nrm{ \nabh\left( \Dh \tilde{\phi}^k\right)  }_2^2  
    \nonumber 
\\
  && 
+ \frac12 \nrm{\Dh \left( \tilde{\phi}^k - \tilde{\phi}^{k-1}  \right) }_2^2
   = F_2 \left( \tilde{\phi}^k \right)-\frac{1}{4}\nrm{ \nabh\left( \Dh \tilde{\phi}^k\right) }_2^2 \ ,
	\label{(6.68)} 
	\end{eqnarray} 
we obtain, with the aid of the estimate  (\ref{(6.65.b)}), 
	\begin{equation} 
F_3\left( \tilde{\phi}^\ell \right) \le   s \, C_{22} \sum_{k=1}^\ell F_3\left( \tilde{\phi}^k \right) + s^2 T \, C_{21} \sum_{k=1}^\ell \sum_{\ell'=1}^k  \nrm{ \tilde{\psi}^{\ell'} }_2^2 
+ C M^2 T (s^2 + h^2)^2 \ ,  
	\label{(6.69-a)}
	\end{equation}
where $C_{22}>0$ is independent of $h$ and $s$. Meanwhile, motivated by the estimate
	\begin{equation} 
  s^2 T \, C_{21} \sum_{k=1}^\ell \sum_{\ell'=1}^k  \nrm{ \tilde{\psi}^{\ell'} }_2^2 
\le  s^2 T \, C_{21} \sum_{k=1}^\ell \sum_{\ell'=1}^k \nrm{ \tilde{\psi}^{\ell'} }_2^2 \le s T^2 \, C_{21} \sum_{k=1}^\ell  \nrm{ \tilde{\psi}^{k} }_2^2  \ ,  
	\end{equation}
which follows from 
	\begin{equation}
\sum_{\ell'=1}^k  \nrm{ \tilde{\psi}^{\ell'} }_2^2 \le \sum_{\ell'=1}^\ell  \nrm{ \tilde{\psi}^{\ell'} }_2^2   ,  \quad \forall \  k \le \ell \ ,
	\label{(6.69-b)}
	\end{equation}
we arrive at       
	\begin{equation} 
F_3\left( \tilde{\phi}^\ell \right) \le   s \, C_{23} \sum_{k=1}^\ell F_3\left( \tilde{\phi}^k \right) 
+ C M^2 T (s^2 + h^2)^2 \ ,  
	\label{(6.69)}
	\end{equation}
where $C_{23}>0$ is independent of $h$ and $s$. Applying a discrete Grownwall inequality gives 
	\begin{equation} 
F_3\left( \tilde{\phi}^\ell \right) \le  C_{24} \left( s^2 + h^2\right)^2 \ ,   \label{(6.70)}
	\end{equation}
which holds provided $s$ is sufficiently small. Note that $C_{24}$ is a positive constant that is dependent upon $T$ (exponentially) and $\Phi$, but is independent of $h$ and $s$.
	\end{proof}

	\begin{corollary}
	\label{cor-global-error-estimate}
Define $\Tilde{\Tilde{\phi}}_{i,j}^k := \Phi^k\left(p_i,p_j,k\, s\right)-\phi_{i,j}^k$.  Then
	\begin{equation}
\nrm{\Tilde{\Tilde{\phi}}^k}_2 +  \nrm{\nabh \left(\Dh \Tilde{\Tilde{\phi}}^k\right)}_2  
\le C\left( s^2 + h^2 \right) \ ,
	\label{estimate-final}
	\end{equation} 
provided $s$ is sufficiently small, for some $C>0$ that is independent of $h$ and $s$.
	\end{corollary}
	 
	\begin{proof}
Estimate (\ref{(6.70)}) gives the discrete $H^3$ estimate for $\tilde{\phi}$.  For the projection $\Phi_N$, we have the following approximation estimate:
	\begin{eqnarray} 
\nrm{\Phi_N - \Phi}_{L^\infty \left(0,T; H^r\right) }  \le  C  h^m \nrm{\Phi}_{L^\infty \left(0,T; H^{m+r}\right) } \ ,
	\label{consistency-Phi-4}
	\end{eqnarray}
for $m,\, r \ge 0$. See for example the references \cite{Boyd2001, CQ1982, HGG2007}.  Combining estimate (\ref{(6.70)}) with the approximation result (\ref{consistency-Phi-4}), we can obtain the estimate (\ref{estimate-final}).
	\end{proof}
	
	\begin{rem}
By virtue of (\ref{estimate-final}) and Lems.~\ref{sobolev-2d-final} and \ref{special-lemma}, along with the estimate (\ref{(6.65.b)}), we immediately get an error estimate of the form 	
	\begin{equation}
\nrm{\Tilde{\Tilde{\phi}}^k}_\infty  \le C\left(h^2+s^2 \right).
	\end{equation} 
	\end{rem}

	\section{Conclusions}\label{sec-conclusions}
	
In this paper, we have established the convergence analysis of an unconditionally energy stable second order accurate finite difference scheme for the sixth-order Modified Phase Field Crystal (MPFC) equation. The parabolic Phase Field Crystal (PFC) equation, which is a mass conserving gradient flow, is obtained as a special case of the MPFC equation. The numerical scheme is based on a second order convex splitting of a discrete psuedo-energy and is semi-implicit.

	\appendix
	
	\section{Tools for Cell-Centered Finite Differences}
	\label{app-tools}

In this first appendix, we define the summation-by-parts formulae, discrete norms, and estimates in two space dimensions that are used to define and analyze our finite difference scheme. With some exceptions, the theory will extend straightforwardly to three-dimensions.  Here we use the same notation and results for 2D cell-centered functions as from~\cite[Sec.~2]{wise09}.  The reader is directed there for all of the missing details.  

For simplicity, we assume that $\Omega = (0,L_x)\times(0,L_y)$.  The framework that we describe has a straightforward extension to three space dimensions. Here we use the notation and results for cell-centered functions from~\cite{wise10,wise09}; see also~\cite{hu09,wang11a}.  The reader is directed to those references for more complete details.  We begin with definitions of grid functions and difference operators needed for our discretization of two-dimensional space.  Let $\Omega = (0,L_x)\times(0,L_y)$, with $L_x = m\cdot h$ and $L_y = n\cdot h$, where $m$ and $n$ are positive integers and $h>0$ is the spatial step size. Define $p_r := (r-\hf)\cdot h$, where $r$ takes on integer and half-integer values.   For any positive integer $\ell$, define $E_\ell = \left\{ p_r \ \middle|\ r=\frac{1}{2},\ldots, \ell+\frac{1}{2}\right\}$, $C_\ell = \left\{p_r \ \middle|\ r=1,\ldots, \ell\right\}$,  $C_{\overline{\ell}} = \left\{ p_r\cdot h\ \middle|\ r=0,\ldots, \ell+1\right\}$. Define the function spaces
	\begin{eqnarray}
{\mathcal C}_{m\times n} &=& \left\{\phi: C_m\times C_n \rightarrow \mathbb{R} \right\},\  {\mathcal C}_{\overline{m}\times\overline{n}} = \left\{\phi: C_{\overline{m}}\times C_{\overline{n}}\rightarrow \mathbb{R} \right\}  ,
    \\
{\mathcal C}_{\overline{m}\times n} &=& \left\{\phi: C_{\overline{m}}\times C_n \rightarrow \mathbb{R} \right\},\ {\mathcal C}_{m\times\overline{n}} = \left\{\phi: C_m\times C_{\overline{n}} \rightarrow \mathbb{R} \right\}  ,
	\\
{\mathcal E}^{\rm ew}_{m\times n} &=& \left\{u: E_m\times C_n \rightarrow\mathbb{R} \right\},\ {\mathcal E}^{\rm ns}_{m\times n} = \left\{v: C_m\times E_n \rightarrow\mathbb{R}  \right\}	 ,
	\\
{\mathcal E}^{\rm ew}_{m\times \overline{n}} &=& \left\{u: E_m\times C_{\overline{n}} \rightarrow\mathbb{R} \right\},\ {\mathcal E}^{\rm ns}_{\overline{m}\times n} = \left\{v: C_{\overline{m}}\times E_n \rightarrow\mathbb{R}  \right\}	 .
  \end{eqnarray}
We use the notation $\phi_{i,j} := \phi\left(p_i,p_j\right)$ for \emph{cell-centered} functions, those in the spaces ${\mathcal C}_{m\times n}$, ${\mathcal C}_{\overline{m}\times n}$, ${\mathcal C}_{m\times\overline{n}}$, or ${\mathcal C}_{\overline{m}\times\overline{n}}$. In component form \emph{east-west edge-centered} functions, those in the spaces ${\mathcal E}^{\rm ew}_{m\times n}$ or ${\mathcal E}^{\rm ew}_{m\times \overline{n}}$, are identified via $u_{i+\hf,j}:=u(p_{i+\hf},p_j)$.  In component form \emph{north-south edge-centered} functions, those in the spaces ${\mathcal E}^{\rm ns}_{m\times n}$, or ${\mathcal E}^{\rm ns}_{\overline{m}\times n}$, are identified via $u_{i+\hf,j}:=u(p_{i+\hf},p_j)$.  The functions of ${\mathcal V}_{m\times n}$ are called \emph{vertex-centered} functions.

We need the weighted 2D grid inner-products $\ciptwo{\, \cdot \,}{\, \cdot \,}$, $\eipew{\, \cdot \,}{\, \cdot \,}$, $\eipns{\, \cdot \,}{\, \cdot \,}$ that are defined in~\cite{wise10,wise09}.  In addition to these, we also need the following one-dimensional inner-products:
	\begin{equation}
\cip{f_{\star,j+\hf}}{g_{\star,j+\hf}} = \sum_{i=1}^m f_{i,j+\hf}g_{i,j+\hf} , \quad \cip{f_{i+\hf,\star}}{g_{i+\hf,\star}} = \sum_{j=1}^n f_{i+\hf,j}g_{i+\hf,j}  ,	
	\end{equation}
where the first is defined for $f,\, g\in{\mathcal E}^{\rm ns}_{m\times n}$, and the second for $f,\, g\in{\mathcal E}^{\rm ew}_{m\times n}$.

The reader is referred to~\cite{wise10,wise09} for the precise definitions of the edge-to-center difference operators  $d_x : {\mathcal E}_{m\times n}^{\rm ew}\rightarrow{\mathcal C}_{m\times n}$ and $d_y : {\mathcal E}_{m\times n}^{\rm ns}\rightarrow{\mathcal C}_{m\times n}$; the $x-$dimension center-to-edge average and difference operators, respectively, $A_x,\, D_x: {\mathcal C}_{\overline{m}\times n}\rightarrow{\mathcal E}_{m\times n}^{\rm ew}$;  the $y-$dimension center-to-edge average and difference operators, respectively, $A_y,\, D_y: {\mathcal C}_{m\times \overline{n}}\rightarrow{\mathcal E}_{m\times n}^{\rm ns}$; and the standard 2D discrete Laplacian, $\Dh : {\mathcal C}_{\overline{m}\times\overline{n}}\rightarrow{\mathcal C}_{m\times n}$.

The summation by parts formula we need from~\cite{wise09} are the following:
	\begin{proposition}\label{sbp-2D-edge}
{\em (Summation-By-Parts:)} If $\phi\in{\mathcal C}_{\overline{m}\times n} \cup{\mathcal C}_{\overline{m}\times\overline{n}}$ and $f\in{\mathcal E}_{m\times n}^{\rm ew}$ then
	\begin{eqnarray}
h^2\, \eipew{D_x \phi}{f} &=& -h^2\, \ciptwo{\phi}{d_x f}  \nonumber
	\\
  &&
 -h\, \cip{A_x\phi_{\hf,\star}}{f_{\hf,\star}}+h\, \cip{A_x\phi_{m+\hf,\star}}{f_{m+\hf,\star}},
	\end{eqnarray}
and if $\phi\in{\mathcal C}_{m\times\overline{n}} \cup {\mathcal C}_{\overline{m}\times\overline{n}}$ and $f\in{\mathcal E}_{m\times n}^{\rm ns}$ then
	\begin{eqnarray}
h^2\, \eipns{D_y\phi}{f} &=& -h^2\, \ciptwo{\phi}{d_y f} \nonumber
	\\
&&-h\, \cip{A_y\phi_{\star,\hf}}{f_{\star,\hf}}+h\, \cip{A_y\phi_{\star,n+\hf}}{f_{\star,n+\hf}}.
	\end{eqnarray}
	\end{proposition}

	\begin{proposition}
\emph{(Discrete Green's Identities:)} Let $\phi,\, \psi\in {\mathcal C}_{\overline{m}\times\overline{n}}$.  Then
	\begin{eqnarray}
&&h^2\, \eipew{D_x\phi}{D_x\psi} + h^2\, \eipns{D_y\phi}{D_y\psi}
	\nonumber
	\\
&=& -h^2\, \ciptwo{\phi}{\Dh\psi}
  - h\, \cip{A_x\phi_{\hf,\star}}{D_x\psi_{  \hf,\star}} + h\, \cip{A_x\phi_{m+\hf,\star}}{D_x \psi_{m+\hf,\star}} 
	\nonumber
	\\
& & \qquad \qquad \qquad 
  - h\, \cip{A_y\phi_{\star,  \hf}}{D_y\psi_{\star,  \hf}} + h\, \cip{A_y \phi_{\star,n+\hf}}{D_y\psi_{\star,n+\hf}} \ , 
	\label{green1stthm-2d}
	\end{eqnarray}
and 
	\begin{eqnarray}
h^2\, \ciptwo{\phi}{\Dh\psi} &=& h^2\, \ciptwo{\Dh\phi}{\psi} \nonumber
	\\
& & + h\, \cip{A_x \phi_{m+\hf,\star}}{D_x \psi_{m+\hf,\star}} 
    - h\, \cip{D_x \phi_{m+\hf,\star}}{A_x \psi_{m+\hf,\star}} \nonumber
	\\
& & - h\, \cip{A_x \phi_{  \hf,\star}}{D_x \psi_{  \hf,\star}} 
    + h\, \cip{D_x \phi_{  \hf,\star}}{A_x \psi_{  \hf,\star}} \nonumber
	\\
& & + h\, \cip{A_y \phi_{\star,n+\hf}}{D_y \psi_{\star,n+\hf}} 
    - h\, \cip{D_y \phi_{\star,n+\hf}}{A_y \psi_{\star,n+\hf}} \nonumber
	\\
& & - h\, \cip{A_y \phi_{\star,  \hf}}{D_y \psi_{\star,  \hf}} 
    + h\, \cip{D_y \phi_{\star,  \hf}}{A_y \psi_{\star,  \hf}} \ .
    	\label{green2ndthm-2d}
	\end{eqnarray} 
	\end{proposition}
 
In this paper we are interested in periodic grid functions.  Specifically, we shall say the cell-centered function $\phi\in {\mathcal C}_{\overline{m}\times\overline{n}}$ is periodic if and only if
	\begin{eqnarray}
\phi_{m+1,  j} = \phi_{1,j}, \quad \phi_{  0,  j} &=& \phi_{m,j}, \quad j = 1,\ldots,n, \label{periodic-bc-1}
	\\
\phi_{  i,n+1} = \phi_{i,1}, \quad \phi_{  i,  0} &=& \phi_{i,n}, \quad i = 0,\ldots,m+1. \label{periodic-bc-2}
	\end{eqnarray}
For such functions, the center-to-edge averages and differences are periodic.  For example, if $\phi\in {\mathcal C}_{\overline{m}\times\overline{n}}$ is periodic, then $A_x\phi_{m+\hf,j} = A_x\phi_{\hf,j}$ and also $D_x\phi_{m+\hf,j} = D_x\phi_{\hf,j}$, for all $j = 0,1, \ldots, n+1$.  We note that the results for periodic functions that are to follow will also hold, in a possibly slightly modified form, when the boundary conditions are taken to be homogeneous Neumann.

We define the following norms for cell-centered functions.  If $\phi\in{\mathcal C}_{m\times n}$, then $\nrm{\phi}_2 := \sqrt{h^2\ciptwo{\phi}{\phi}}$, and we define $\nrm{\nabh\phi}_2$, where $\phi\in {\mathcal C}_{\overline{m}\times\overline{n}}$, to mean
	\begin{equation}
\nrm{\nabh\phi}_2 := \sqrt{h^2\, \eipew{D_x\phi}{D_x\phi} +h^2\, \eipns{D_y\phi}{D_y\phi} \ } \ .
	\end{equation}
We will use the following discrete Sobolev-type norms for grid functions $\phi\in {\mathcal C}_{\overline{m}\times\overline{n}}$: $\nrm{\phi}_{0,2} := \nrm{\phi}_2$ and
	\begin{equation}
\nrm{\phi}_{1,2} := \sqrt{\nrm{\phi}_2^2+\nrm{\nabh\phi}_2^2 \ } \ ,\quad \nrm{\phi}_{2,2} := \sqrt{\nrm{\phi}_2^2+\nrm{\nabh\phi}_2^2+\nrm{\Dh\phi}_2^2 \  } \ . 
	\end{equation}
In addition, we introduce the following discrete $L^4$ and $L^\infty$ norms: for any $\phi\in{\mathcal C}_{m\times n}$ define
        \begin{equation}
\nrm{\phi}_4 := \Big( h^2\, \ciptwo{\phi^4}{{\bf 1}}\Big)^{1/4} \quad \mbox{and}\quad \nrm{\phi}_\infty = \max_{1\le i\le m \atop 1\le j\le n}\left|\phi_{i,j}\right|  . \label{(5.25)}
        \end{equation}
And for $\phi\in {\mathcal C}_{\overline{m}\times\overline{n}}$ define
	\begin{equation}
\nrm{\nabh\phi}_4 := \Big(h^2\, \eipew{\left(D_x\phi\right)^4}{{\bf 1}} +h^2\, \eipns{\left(D_y\phi\right)^4}{{\bf 1}}\Big)^{1/4}\ .
	\end{equation}

Some discrete Sobolev-type inequalities for two-dimensional grid functions are needed in the analysis in later sections. The following results are recalled; the detailed proofs can be found in~\cite{wise09, wang11a}.

	\begin{lemma}\label{bramble}
Suppose that $\phi\in {\mathcal C}_{\overline{m}\times\overline{n}}$.  Then,
	\begin{equation}
\nrm{\phi}_4 \le C_1 \nrm{\phi}_{1,2}\ ,\quad C_1:= \left(2\max\left[\max\left\{\frac{1}{L_x},L_x\right\},\max\left\{\frac{1}{L_y},L_y\right\} \right]\right)^{1/4}\ .
	\end{equation}
	\end{lemma}

	\begin{lemma}
	\label{sobolev-2d-final}
Suppose that $\phi\in {\mathcal C}_{\overline{m}\times\overline{n}}$ is periodic.  Then, for any $i\in \left\{1,2,\ldots ,m\right\}$ and any $j\in \left\{1,2,\ldots ,n\right\}$,
	\begin{equation}
\left|\phi_{i,j} \right|^2 \le C_2\nrm{\phi}^2_{2,2} \ ,\quad C_2:=4\max\left\{\frac{1}{L_xL_y},\frac{L_x}{L_y},\frac{L_y}{L_x},\frac{L_xL_y}{2}\right\} \ .
	\end{equation}
Hence $\nrm{\phi}^2_\infty \le C_2\nrm{\phi}^2_{2,2}$.
	\end{lemma}
	
	\begin{lemma}
	\label{laplacian-estimate}
Suppose that $\phi\in {\mathcal C}_{\overline{m}\times\overline{n}}$ is periodic.  Define	
	\begin{equation}
S:= h^2\sum_{i'=0}^{m}\sum_{j'=0}^n w^m_{i'}w^n_{j'} \left|D_y\left(D_x\phi\right)_{i'+\hf,j'+\hf}\right|^2 ,
	\end{equation}
where
	\begin{equation}\label{w-function}
w^\ell_k := \left\{
	\begin{array}{lll}
1 & \mbox{if} & k\in\left\{1,2,\ldots ,\ell-1\right\}
	\\
\hf & \mbox{if} & k\in\left\{0,\ell\right\}
	\end{array}
\right.  .
	\end{equation}
Then $S = h^2\ciptwo{\Dh^x\phi}{\Dh^y\phi}$, where $\Dh^x := d_x D_x$ and $\Dh^y := d_y D_y$ are the 3-point discrete lapacian operators in the $x$- and $y$-directions, respectively~\cite{wise09}.  And, since $S \ge 0$, we have
	\begin{equation}\label{laplace-estimate-x-y}
h^2\,\ciptwo{\Dh^x\phi}{\Dh^x\phi} \le h^2\,\ciptwo{\Dh\phi}{\Dh\phi}\quad \mbox{and}\quad  h^2\,\ciptwo{\Dh^y\phi}{\Dh^y\phi} \le h^2\,\ciptwo{\Dh\phi}{\Dh\phi}.
	\end{equation}
	\end{lemma}

Consider the space
	\begin{equation}
	\label{mean-zero-space}
H := \left\{\phi\in{\mathcal C}_{m\times n} | \ciptwo{\phi}{{\bf 1}}=0 \right\},
	\end{equation}
and equip this space with the bilinear form
	\begin{equation}
\ciptwo{\phi_1}{\phi_2}_{-1} := \eipew{D_x \psi_1}{D_x \psi_2}  + \eipns{D_y \psi_1}{D_y \psi_2} ,
	\end{equation}
for any $\phi_1,\, \phi_2\in H$, where $\psi_i\in{\mathcal C}_{\overline{m}\times\overline{n}}$ is the unique solution to 
	\begin{equation}
-\Dh \psi_i  = \phi_i,\quad \psi_i\mbox{ periodic},\quad \ciptwo{\psi_i}{{\bf 1}} = 0.
	\end{equation}
	
	\begin{lemma}
	\label{h-l-inner-product}
The bilinear form $\ciptwo{\phi_1}{\phi_2}_{-1}$ is an inner product on the space $H$.  Moreover,
	\begin{equation}
\ciptwo{\phi_1}{\phi_2}_{-1} = -\ciptwo{\phi_1}{\Dh^{-1}\left(\phi_2\right)} = -\ciptwo{\Dh^{-1}\left(\phi_1\right)}{\phi_2}.
	\end{equation}
Thus
	\begin{equation}
\nrm{\phi}_{-1} :=\sqrt{h^2\ciptwo{\phi}{\phi}_{-1}}
	\end{equation}
defines a norm on $H$.
	\end{lemma}

	\begin{lemma}
	\label{embedding}
Suppose that $\phi\in {\mathcal C}_{\overline{m}\times\overline{n}}$ is periodic and set $\bar{\phi} = \frac{1}{m \cdot n} \ciptwo{\phi}{{\bf 1}}$.  Then
	\begin{equation}
\nrm{\phi - \bar{\phi}}_2 \le C_3\nrm{\nabh\phi}_2, 
	\end{equation}
where $C_3>0$ is a constant that only depends upon $L_x$ and $L_y$.  Furthermore,
	\begin{equation}
\nrm{\phi-\bar{\phi}}_4 \le C_4\nrm{\nabh\phi}_2 ,\quad \nrm{\nabh\phi}_4 \le C_4\nrm{\Dh\phi}_2 .
	\end{equation}
where $C_4 := C_1\sqrt{C_3^2+1}$.
	\end{lemma}

	\section{Consistency Analysis of the Second Order Numerical Scheme} 
	\label{app-consistency}

In this appendix we give a detailed derivation the local truncation error estimate (\ref{truncation-error}).  We establish the results for vertex-centered grid functions, rather than cell-centered functions, as the indexing becomes simpler. 

	\subsection{Proof of Estimate (\ref{truncation-error})} 
	\label{sub-app-truncation-error}

The following three results will be used to establish (\ref{truncation-error}).
	\begin{proposition}
	\label{Prop-A.0}
For $f \in H^3 (0,T)$, we have 
	\begin{equation}
\nrm{\tau^t f }_{L_s^2 (0,T)}  \le C s^m \nrm{f}_{H^{m+1} (0,T)}  ,  \quad 
   \mbox{with} \quad \tau^t f ^k = \frac{f^{k+1} - f^k}{s} - f' (t^{k+1/2}) ,   
	\label{est-1d-time-1}
	\end{equation}
for $0 \le m \le 2$, where $C$ only depends on $T$, $\nrm{\ \cdot \ }_{L_s^2 (0,T)}$ is a discrete $L^2$ norm (in time) given by $\nrm{g}_{L_s^2 (0,T)} = \sqrt{s \sum_{k=0}^{[T/s]-1} \left(g^k\right)^2}$.  
   \end{proposition}

	\begin{proposition}
	\label{Prop-A.1}
For $f \in H^2 (0,T)$, we have 
	\begin{eqnarray}
  &&
  \nrm{D_{t/2}^2 f }_{L_s^2 (0,T)}  := 
  \sqrt{s \sum_{k=0}^{[T/s]-1} \left( D_{t/2}^2 f^{k+1/2} \right)^2}
 \le C \nrm{f}_{H^2 (0,T)}  ,  \label{est-1d-time-2-1} 
\\
  &&  
  \mbox{with}  \quad   
 D_{t/2}^2 f^{k+1/2}  = \frac{4 \left( f^{k+1} - 2 f (\  \cdot\ , t^{k+1/2} ) + f^k \right)}{s^2} , 
   \nonumber 
\\
  &&
  \nrm{D_t^2 f }_{L_s^2 (0,T)}  := 
  \sqrt{s \sum_{k=0}^{[T/s]-1} \left( D_t^2 f^k \right)^2}
 \le C \nrm{f}_{H^2 (0,T)}  ,  \label{est-1d-time-2-2} 
\\
  &&  
  \mbox{with}  \quad  
 D_t^2 f^k  = \frac{f^{k+1} - 2 f^k + f^{k-1} }{s^2} ,  \nonumber 
	\end{eqnarray}
where $C$ only depends on $T$. 
   \end{proposition}

The proof of Proposition~\ref{Prop-A.0}-\ref{Prop-A.1} is based on the integral form of the Taylor expansion in time. The details are skipped for the sake of brevity. 
In the spatial discretization, the following proposition gives a corresponding $O(h^2)$ truncation error bound. 

	\begin{proposition}
	\label{Prop-A.2}
If $f \in {\cal B}^{N/2}$ has a regularity $f \in H^{8}_{per} (\Omega)$, we have 
	\begin{equation} 
\nrm{\Delta^k f - \Delta_h^k f}_{L_h^2 (\Omega)}  
\le C h^2 \nrm{f}_{H^{2 + 2k} (\Omega)} ,  \quad 
  \mbox{for}  \quad k= 1, 2, 3, 
	\label{est-2d-1}
	\end{equation} 
where $C$ only depends on $L_0$ and $\nrm{g}_{L_h^2 (\Omega)} = \sqrt{h^2 \sum_{i,j=0}^{N-1} g_{i,j}^2}$.
	\end{proposition}
	
The key point of this proposition is that the projection approximation solution 
$\Phi_N \in {\cal B}^{N/2}$ so that an aliasing error is avoided in its centered difference 
approximation. This consistency analysis can be carried out by a detailed Fourier 
expansion of both $\Delta_h^k \Phi_N$ and $\Delta \Phi_N$ at the discrete level, 
and the comparison of the corresponding discrete Fourier coefficients leads to the 
above estimates, following a similar methodology as in \cite{wang11a}. 
The details are skipped for brevity.

Observe that the ${\mathcal O}\left(s^2\right)$ correction in the definition of (\ref{consistency-Phi-3}) is added so that a higher order consistency between $\Psi_N$ at $t^{k+1/2}$ and $(\Phi_N^{k+1} - \Phi_N^k)/s$ can be derived. Looking at the time derivative of the projection operator, we observe that
	\begin{equation} 
\frac{\partial^k}{\partial t^k}  \Phi_N (\x, t) = \frac{\partial^k}{\partial t^k}  {\cal P}_N \Phi (\x,t)  = {\cal P}_N  \frac{\partial^k  \Phi (\x,t) }{\partial t^k} \ . 
	\label{consistency-Phi-5}
	\end{equation} 
In other words, $\partial_t^k \Phi_N$ is the truncation of $\partial_t^k \Phi$ for any $k \ge 0$, since projection and differentiation commute.
This in turn implies an accurate approximation of the corresponding temporal derivative, at any fixed time:  
	\begin{equation} 
\nrm{\partial_t^k \left( \Phi_N - \Phi \right)}_{H^r} \le  C  h^m \nrm{\partial_t^k \Phi}_{H^{m+r}} \ , 
	\label{consistency-Phi-6}
	\end{equation}
for $m,\, r \ge 0$, and  $0 \le k \le 2$.

Since the exact solution of the MPFC equation has the regularity (\ref{MPFC-regularity}), the approximation estimates (\ref{consistency-Phi-4})-(\ref{consistency-Phi-6}) imply the same regularity for the projection solution $\Phi_N$, $\Psi_N$, by taking $m=0$: 
	\begin{equation}
\left\|  \Phi_N  \right\|_{H^r}  \le  C  \left\|  \Phi \right\|_{H^r} ,  \quad 
 \left\|  \partial_t^k \left( \Phi_N - \Phi \right) \right\|_{H^r}
  \le  C  \left\|  \partial_t^k \Phi \right\|_{H^r} ,
	\label{MPFC-regularity-projection}
	\end{equation} 
at any fixed time. 	
	
We define the following quantities: 
 \begin{equation}
 \begin{array}{rclrclrcl} 
F_1^{k+1/2} &\!\!\!=\!\!\!& \frac{\Psi_N^{k+1} - \Psi_N^k}{s} \ ,  & F_{1e}^{k+1/2} &\!\!\!=\!\!\!& \left( \partial_t^2 \Phi_N \right) (\  \cdot\ , t^{k+1/2} ) \ , 
	\\
F_2^{k+1/2} &\!\!\!=\!\!\!&  \frac{\Phi_N^{k+1} - \Phi_N^k}{s} \ , & F_{2e}^{k+1/2} &\!\!\!=\!\!\!& 
\left( \partial_t \Phi_N \right) (\  \cdot\ , t^{k+1/2} )  \ , 
	\\
F_3^{k+1/2} &\!\!\!=\!\!\!& \Dh \left(  ( \Phi_N^3 )^{k+1/2}  \right) \ , & F_{3e}^{k+1/2} &\!\!\!=\!\!\!& \Delta \left(  (\Phi_N)^3  \right) (\  \cdot\ , t^{k+1/2} )\ , 
        \\
F_4^{k+1/2} &\!\!\!=\!\!\!& \Dh \left(  \Phi_N^{k+1/2}  \right) \ , & F_{4e}^{k+1/2} &\!\!\!=\!\!\!& \Delta \left(  \Phi_N  \right) (\  \cdot\ , t^{k+1/2} )\ ,
  	\\
F_5^{k+1/2} &\!\!\!=\!\!\!& \Dh^2 \left( \frac32 \Phi_N^k 
- \frac12 \Phi_N^{k-1} \right) \ , & F_{4e}^{k+1/2} &\!\!\!=\!\!\!& 
  \left( \Delta^2 \Phi_N \right) (\  \cdot\ , t^{k+1/2} ) \ , 
	\\
F_6^{k+1/2} &\!\!\!=\!\!\!& \Dh^3 \Phi_N^{k+1/2} \ , & F_{6e}^{k+1/2} &\!\!\!=\!\!\!& 
\left( \Delta^3 \Phi_N \right) (\  \cdot\ , t^{k+1/2} ) \ ,  
           \\
F_7^{k+1/2} &\!\!\!=\!\!\!& \frac{ \Psi_N^{k+1} + \Psi_N^k}{2}\ . & & &    
	\end{array}
	\label{truncation-decomp-1}
	\end{equation}
Moreover, the corresponding values for the exact solution are denoted by 
\begin{eqnarray}
  &&
  F_{1en}^{k+1/2} = \partial_t^2 \Phi (\  \cdot\ , t^{k+1/2} ) ,   \quad 
  F_{2en}^{k+1/2} = \partial_t \Phi (\  \cdot\ , t^{k+1/2} ) ,   \nonumber 
\\
  && 
  F_{3en}^{k+1/2} = \left( \Delta \Phi^3 \right) (\  \cdot\ , t^{k+1/2} ) ,  \quad 
  F_{4en}^{k+1/2} = \left( \Delta \Phi \right) (\  \cdot\ , t^{k+1/2} ) , \nonumber 
\\
  &&   
  F_{5en}^{k+1/2} = \left( \Delta^2 \Phi \right) (\  \cdot\ , t^{k+1/2} ) ,  \quad  
  F_{6en}^{k+1/2} = \left( \Delta^3 \Phi \right) (\  \cdot\ , t^{k+1/2} ) .
\end{eqnarray}
Note that all these quantities are defined on the numerical grid (in space) point-wise. 

  First we look at the first order time derivative term, $F_2$, $F_{2e}$ and $F_{2en}$. 
A direct application of Proposition~\ref{Prop-A.0}
indicates that (by taking $m=2$): 
       \begin{eqnarray} 
\nrm{ F_2 - F_{2e} }_{L_s^2 (0,T)}  \le C s^2 \nrm{ \Phi_N}_{H^{3} (0,T)} 
   \le C s^2 \nrm{ \Phi}_{H^{3} (0,T)} ,
	\label{truncation-est-2nd-2-1} 
        \end{eqnarray} 
for each fixed grid point $(i,j)$, in which the second part of (\ref{MPFC-regularity-projection}) 
was used in the second step. Meanwhile, the approximation estimate 
(\ref{consistency-Phi-6}) yields (with $k=1$, $m=2$):  
\begin{eqnarray} 
  \left\|  F_{2e}^{k+1/2} - F_{2en}^{k+1/2}  \right\| 
  \le C h^2 \left\|  \partial_t \Phi  \right\|_{H^2} .
  \label{truncation-est-2nd-2-2} 
        \end{eqnarray} 
Therefore, a careful calculation shows that a combination of the above two estimates 
results in 
        \begin{eqnarray} 
\nrm{ F_2 - F_{2en} }_{L_s^2 \left(0,T; L_h^2(\Omega)\right)}  
  \le C ( s^2 + h^2 )  \left( \nrm{ \Phi}_{H^{3} (0,T; L^2)} 
  +  \left\|  \Phi  \right\|_{W^{1,\infty} (0,T; H^2)}  \right) . 
	\label{truncation-est-2nd-2-3} 
        \end{eqnarray} 
        
  Similar analysis can be applied to the second order time derivative terms. 
The construction (\ref{consistency-Phi-3}) for the approximate solution $\Psi_N$ gives 
\begin{eqnarray} 
  F_1^{k+1/2} = \frac{\partial_t \Phi_N^{k+1} - \partial_t \Phi_N^k}{s} 
  - \frac{s^2}{12} \cdot \frac{\partial_t^3 \Phi_N^{k+1} - \partial_t^3 \Phi_N^k}{s}  
  := F_{11}^{k+1/2} - \frac{s^2}{12} F_{12}^{k+1/2} ,  \label{truncation-est-2nd-1-1} 
\end{eqnarray}   
in which $F_{11}$ and $F_{12}$ are finite difference (in time) approximation to 
$\partial_t^2 \Phi_N$, $\partial_t^4 \Phi_N$, respectively. In more detail, if we denote 
$F_{11e}^{k+1/2} = \partial_t^2 \Phi_N (\  \cdot\ , t^{k+1/2} )$, 
$F_{12e}^{k+1/2} = \partial_t^4 \Phi_N (\  \cdot\ , t^{k+1/2} )$, the following estimates 
can be derived by using Proposition~\ref{Prop-A.0} (with $m=2$ and $m=0$): 
       \begin{eqnarray} 
       &&
    \nrm{ F_{11} - F_{11e} }_{L_s^2 (0,T)}  \le C s^2 \nrm{ \Phi_N}_{H^{4} (0,T)} 
   \le C s^2 \nrm{\Phi}_{H^{4} (0,T)} ,   \label{truncation-est-2nd-1-2} 
\\
    &&
    \nrm{ F_{12} - F_{12e} }_{L_s^2 (0,T)}  \le C \nrm{ \Phi_N}_{H^{4} (0,T)} 
   \le C \nrm{\Phi}_{H^{4} (0,T)} ,   \label{truncation-est-2nd-1-3}       
        \end{eqnarray} 
for each fixed grid point $(i,j)$. Again, the approximation estimate 
(\ref{consistency-Phi-6}) gives (with $k=2$, $m=2$):  
\begin{eqnarray} 
  \left\|  F_{1e}^{k+1/2} - F_{1en}^{k+1/2}  \right\| 
  \le C h^2 \left\|  \partial_t^2 \Phi  \right\|_{H^2} .
  \label{truncation-est-2nd-1-4} 
        \end{eqnarray} 
A combination of (\ref{truncation-est-2nd-1-1})-(\ref{truncation-est-2nd-1-4}) leads to 
        \begin{eqnarray} 
\nrm{ F_1 - F_{1en} }_{L_s^2 \left(0,T; L_h^2(\Omega)\right)}  
  \le C ( s^2 + h^2 )  \left( \nrm{ \Phi}_{H^4 (0,T; L^2)} 
  +  \left\|  \Phi  \right\|_{W^{2,\infty} (0,T; H^2)}  \right) . 
	\label{truncation-est-2nd-1-5} 
        \end{eqnarray} 

For the convex diffusion term $F_4$, $F_{4e}$ and $F_{4en}$, we start from an 
application of Prop.~\ref{Prop-A.2} 
(recall that $\Phi_N^{k+1/2} = \frac{\Phi_N^{k+1} + \Phi_N^k}{2}$): 
\begin{eqnarray} 
  \nrm{F_4^{k+1/2} -  \Delta \left(  \Phi_N^{k+1/2}  \right) }_{L_h^2 (\Omega)}  
  &\le& C h^2  \nrm{\Phi_N^{k+1/2}}_{H^{4}(\Omega)}   \nonumber 
\\
  &\le& C h^2  \nrm{\Phi_N}_{L^{\infty} (0,T; H^4)} .   
  	\label{truncation-est-2nd-4-1} 
        \end{eqnarray} 	
Meanwhile, a comparison between $\Phi_N^{k+1/2}$ and $ \Phi_N (\  \cdot\ , t^{k+1/2} )$ 
shows that 
\begin{eqnarray} 
  \Phi_N^{k+1/2} - \Phi_N (\  \cdot\ , t^{k+1/2} ) 
  = \frac12 s^2 D_{t/2}^2 \Phi_N^{k+1/2} . \label{truncation-est-2nd-4-2} 
\end{eqnarray} 
On the other hand, an application of Prop.~\ref{Prop-A.1} gives  
\begin{eqnarray}
  \nrm{D_{t/2}^2 \Delta \Phi_N}_{L_s^2 (0,T)}   
 \le C \nrm{\Delta \Phi_N}_{H^2 (0,T)}  ,  \label{truncation-est-2nd-4-3}  
\end{eqnarray}
at each fixed grid $(i,j)$. As a result of 
(\ref{truncation-est-2nd-4-1})-(\ref{truncation-est-2nd-4-3}), we get 
\begin{eqnarray} 
  \nrm{ F_4 - F_{4e} }_{L_s^2 \left(0,T; L_h^2(\Omega)\right)}  
  &\le& C ( s^2 + h^2 )  \left( \nrm{ \Phi_N}_{L^\infty (0,T; H^4)} 
  +  \left\|  \Phi_N  \right\|_{H^2 (0,T; H^2)}  \right)   \nonumber 
\\
   &\le& C ( s^2 + h^2 )  \left( \nrm{ \Phi}_{L^\infty (0,T; H^4)} 
  +  \left\|  \Phi  \right\|_{H^2 (0,T; H^2)}  \right) .  
	\label{truncation-est-2nd-4-4} 
        \end{eqnarray} 
The approximation estimate of $F_{4e}$ to $F_{4en}$ is straightforward, from 
(\ref{consistency-Phi-4}) (with $m=2$): 
\begin{eqnarray} 
  \nrm{ F_{4e} - F_{4en}}_{L_s^2 \left(0,T; L_h^2(\Omega)\right)}  
   \le C h^2  \nrm{ \Phi}_{L^\infty (0,T; H^4)} .  
	\label{truncation-est-2nd-4-5} 
\end{eqnarray} 
Consequently, we arrive at 
\begin{eqnarray} 
  \nrm{ F_4 - F_{4en} }_{L_s^2 \left(0,T; L_h^2(\Omega)\right)}  
  \le C ( s^2 + h^2 )  \left( \nrm{ \Phi}_{L^\infty (0,T; H^4)} 
  +  \left\|  \Phi  \right\|_{H^2 (0,T; H^2)}  \right) .  
	\label{truncation-est-2nd-4-6} 
\end{eqnarray} 

The other convex diffusion terms $F_6$, $F_{6e}$ and $F_{6en}$ can be analyzed 
in the same way. The details are skipped for simplicity. 
\begin{eqnarray} 
  \nrm{ F_6 - F_{6en} }_{L_s^2 \left(0,T; L_h^2(\Omega)\right)}  
  \le C ( s^2 + h^2 )  \left( \nrm{ \Phi}_{L^\infty (0,T; H^8)} 
  +  \left\|  \Phi  \right\|_{H^2 (0,T; H^6)}  \right) .  
	\label{truncation-est-2nd-6} 
\end{eqnarray}  

The analysis for the concave diffusion terms $F_5$, $F_{5e}$ and $F_{5en}$ is 
similar to that of the convex diffusion term; yet more details are involved. 
An application of Prop.~\ref{Prop-A.2} gives  
\begin{eqnarray} 
  \nrm{F_5^{k+1/2} -  \Delta^2 \left(  \frac32 \Phi_N^k - \frac12 \Phi_N^{k-1}  
  \right) }_{L_h^2 (\Omega)}  
  &\le& C h^2  \nrm{\left(  \frac32 \Phi_N^k - \frac12 \Phi_N^{k-1}  
  \right)}_{H^{6}(\Omega)}   \nonumber 
\\
  &\le& C h^2  \nrm{\Phi_N}_{L^{\infty} (0,T; H^6)} .   
  	\label{truncation-est-2nd-5-1} 
        \end{eqnarray} 	
Meanwhile, a comparison between $\frac32 \Phi_N^k - \frac12 \Phi_N^{k-1}$ 
and $ \Phi_N (\  \cdot\ , t^{k+1/2} )$ reveals that 
\begin{eqnarray} 
  \left(  \frac32 \Phi_N^k - \frac12 \Phi_N^{k-1} \right) - \Phi_N (\  \cdot\ , t^{k+1/2} ) 
  = \frac12 s^2 D_{t/2}^2 \Phi_N^{k+1/2} 
  - \frac12 s^2 D_t^2 \Phi_N^k  . \label{truncation-est-2nd-5-2} 
\end{eqnarray} 
Similarly, applications of Prop.~\ref{Prop-A.1} imply  
\begin{eqnarray}
  && 
  \nrm{D_{t/2}^2 \Delta^2 \Phi_N}_{L_s^2 (0,T)}   
 \le C \nrm{\Delta^2 \Phi_N}_{H^2 (0,T)}  ,  \nonumber 
\\
  &&
  \nrm{D_t^2 \Delta^2 \Phi_N}_{L_s^2 (0,T)}   
 \le C \nrm{\Delta^2 \Phi_N}_{H^2 (0,T)}  , 
 \label{truncation-est-2nd-5-3}  
\end{eqnarray}
at each fixed grid $(i,j)$. Then we obtain 
\begin{eqnarray} 
  \nrm{ F_5 - F_{5e} }_{L_s^2 \left(0,T; L_h^2(\Omega)\right)}  
  &\le& C ( s^2 + h^2 )  \left( \nrm{ \Phi_N}_{L^\infty (0,T; H^6)} 
  +  \left\|  \Phi_N  \right\|_{H^2 (0,T; H^4)}  \right)   \nonumber 
\\
   &\le& C ( s^2 + h^2 )  \left( \nrm{ \Phi}_{L^\infty (0,T; H^6)} 
  +  \left\|  \Phi  \right\|_{H^2 (0,T; H^4)}  \right) .  
	\label{truncation-est-2nd-5-4} 
        \end{eqnarray} 
The approximation estimate of $F_{5e}$ to $F_{5en}$ can be derived in the 
same manner: 
\begin{eqnarray} 
  \nrm{ F_{5e} - F_{5en}}_{L_s^2 \left(0,T; L_h^2(\Omega)\right)}  
   \le C h^2  \nrm{ \Phi}_{L^\infty (0,T; H^6)} .  
	\label{truncation-est-2nd-5-5} 
\end{eqnarray} 
That gives the consistency estimate for the concave diffusion: 
\begin{eqnarray} 
  \nrm{ F_5 - F_{5en} }_{L_s^2 \left(0,T; L_h^2(\Omega)\right)}  
  \le C ( s^2 + h^2 )  \left( \nrm{ \Phi}_{L^\infty (0,T; H^6)} 
  +  \left\|  \Phi  \right\|_{H^2 (0,T; H^4)}  \right) .  
	\label{truncation-est-2nd-5-6} 
\end{eqnarray}

Next we look at the nonlinear term.  A direct application of Prop.~\ref{Prop-A.2} indicates that 
	\begin{eqnarray} 
\nrm{F_3^{k+1/2} - \Delta \left(  ( \Phi_N^3 )^{k+1/2}  \right) }_{L_h^2 (\Omega)}  
&\le& C h^2 \nrm{ ( \Phi_N^3 )^{k+1/2} }_{H^{4}(\Omega) }  \nonumber 
\\
  &\le& 
   C h^2  \left(  \left\|  \Phi_N^{k+1}  \right\|_{H^4 (\Omega)}^3 
   +   \left\|  \Phi_N^k  \right\|_{H^4 (\Omega)}^3 \right)  , 
  	\label{truncation-est-2nd-3-1}
	\end{eqnarray} 
in which a product expansion and a Sobolev imbedding are used in the second step. 
Subsequently, we need to compare $(\Phi_N^3)^{k+1/2}$ and 
$(\Phi_N^3 ) (\  \cdot\ , t^{k+1/2} )$ and derive an estimate. A careful 
calculation reveals that 
\begin{eqnarray} 
  (\Phi_N^3)^{k+1/2} - (\Phi_N^3 ) (\  \cdot\ , t^{k+1/2} ) 
  &=& 
  \frac12 \left( (\Phi_N^{k+1})^2 + (\Phi_N^k)^2 \right) 
  \cdot \frac18 s^2 \left( D_{t/2}^2 \Phi_N \right)^{k+1/2}  \nonumber 
\\
  && 
  + \frac18 s^2 \left( D_{t/2}^2 \Phi_N^2 \right)^{k+1/2}  
  \cdot \Phi_N (\  \cdot\ , t^{k+1/2} )  .  \label{truncation-est-2nd-3-2}
\end{eqnarray}  
Meanwhile, by the following observation of a nonlinear expansion: 
\begin{equation} 
  \Delta (f g h) = f g \Delta h + f h \Delta g + gh \Delta f 
  + 2 f \nabla g \nabla h  + 2 g \nabla f \nabla h   
  + 2 h \nabla f \nabla g ,   \label{truncation-est-2nd-3-3}
\end{equation}  
we obtain 
\begin{eqnarray}  
  &&
  \left\|  (\Phi_N^3)^{k+1/2} - (\Phi_N^3 ) (\  \cdot\ , t^{k+1/2} )  \right\|_{L_h^2 (\Omega)} 
  \nonumber 
\\
  &\le& 
  C \tilde{C} \left( \tilde{C} + 1 \right) s^2 
  \left(  \left\| \left( D_{t/2}^2 \Phi_N \right)^{k+1/2}   \right\|_{H^2} 
  +   \left\| \left( D_{t/2}^2 \Phi_N^2 \right)^{k+1/2}   \right\|_{H^2}  \right) ,  \nonumber 
\\
  && \mbox{with} \quad 
  \tilde{C} = \left\| \Phi_N  \right\|_{L^\infty (0,T; W^{2,\infty} (\Omega) )}   
  \le C \left\| \Phi_N  \right\|_{L^\infty (0,T; H^4 (\Omega) )} . 
  \label{truncation-est-2nd-3-4}
\end{eqnarray}  
Subsequently, applications of Prop.~\ref{Prop-A.1} imply  
\begin{eqnarray}
  && 
  \nrm{D_{t/2}^2 \Phi_N}_{L_s^2 (0,T; H^2)}   
 \le C \nrm{\Phi_N}_{H^2 (0,T; H^2)}  ,  \label{truncation-est-2nd-3-5}
\\
  &&
  \nrm{D_{t/2}^2 (\Phi_N^2) }_{L_s^2 (0,T; H^2)}   
 \le C \nrm{\Phi_N^2}_{H^2 (0,T; H^2)}  . 
 \label{truncation-est-2nd-3-6}  
\end{eqnarray}
Note that the second estimate is involved with a nonlinear term $\Phi_N^2$. 
A detailed expansion in its first and second order time derivatives shows that
\begin{eqnarray} 
  \partial_t (\Phi_N^2) = 2 \Phi_N \partial_t \Phi_N ,  \quad  
  \partial_t^2 (\Phi_N^2) = 2 \Phi_N \partial_t^2 \Phi_N  
  + 2 ( \partial_t \Phi_N )^2 ,  \label{truncation-est-2nd-3-7}  
\end{eqnarray}
which in turn leads to
\begin{eqnarray} 
  \nrm{\Phi_N^2}_{H^2 (0,T)}  
  &\le& 
  C \left(  \nrm{\Phi_N}_{L^\infty (0,T)}  \cdot  \nrm{\Phi_N}_{H^2 (0,T)}
  +  \nrm{\Phi_N}_{W^{1,4} (0,T)}^2 \right)  \nonumber 
\\
  &\le& 
  C  \nrm{\Phi_N}_{H^2 (0,T)}^2 ,   \label{truncation-est-2nd-3-8}  
\end{eqnarray}
at each fixed grid point $(i,j)$, with a 1-D Sobolev imbedding applied at the last step.
Going back to (\ref{truncation-est-2nd-3-6}) gives 
\begin{eqnarray} 
  \nrm{D_{t/2}^2 (\Phi_N^2) }_{L_s^2 (0,T; H^2)}   
 \le C \nrm{\Phi_N}_{H^2 (0,T; H^2)}^2  . 
 \label{truncation-est-2nd-3-9}  
\end{eqnarray}
Therefore, a substitution of (\ref{truncation-est-2nd-3-5}) and 
(\ref{truncation-est-2nd-3-9}) into (\ref{truncation-est-2nd-3-4}) yields 
\begin{eqnarray}  
  &&
  \left\|  \Delta \left( (\Phi_N^3)^{k+1/2} 
  - (\Phi_N^3 ) (\  \cdot\ , t^{k+1/2} )  \right) \right\|_{L_s^2 (0, T; L_h^2 (\Omega) )} 
  \nonumber 
\\
  &\le& 
  C \left( \left\| \Phi_N  \right\|_{L^\infty (0,T; H^4 (\Omega) )}^2 
  \cdot \nrm{\Phi_N}_{H^2 (0,T; H^2)}^2 + 1  \right) s^2 . 
  \label{truncation-est-2nd-3-10}
\end{eqnarray}  
For the comparison between $F_{3e}$ and $F_{3en}$, we cannot apply 
(\ref{consistency-Phi-4}) directly, since $\Phi_N^3$ is not in ${\cal B}^{N/2}$. 
We observe the difference between $\Phi_N^3$ and $\Phi^3$ is given by 
\begin{eqnarray} 
  \Phi_N^3 - \Phi^3  =  \left( \Phi_N - \Phi  \right) 
   \left( \Phi_N^2 + \Phi_N \Phi + \Phi^2  \right)  .  
   \label{truncation-est-2nd-3-11}
\end{eqnarray}  
As a result, taking a Laplacian operator to the above terms, applying the nonlinear
expansion (\ref{truncation-est-2nd-3-3}), we arrive at 
\begin{eqnarray} 
  &&
  \left\|  F_{3e}^{k+1/2} - F_{3en}^{k+1/2}  \right\|_{L_h^2 (\Omega)}  \nonumber 
\\
  &\le& 
  C  \left(  \left\| \Phi_N  \right\|_{L^\infty (0,T; W^{2,\infty} (\Omega) )}^2 
  + \left\| \Phi  \right\|_{L^\infty (0,T; W^{2,\infty} (\Omega) )}^2  \right) 
  \cdot  \left\| \Phi_N - \Phi \right\|_{L^\infty (0,T; H^2 (\Omega) )}   \nonumber 
\\
  &\le& 
  C  h^2 \left\| \Phi \right\|_{L^\infty (0,T; H^4 (\Omega) )}^3 . 
   \label{truncation-est-2nd-3-12}
\end{eqnarray} 
Furthermore, a combination of (\ref{truncation-est-2nd-3-1}), 
(\ref{truncation-est-2nd-3-10}) and (\ref{truncation-est-2nd-3-12}) leads to the 
consistency estimate of the nonlinear term 
\begin{eqnarray} 
  &&
  \left\|  F_3  - F_{3en} \right\|_{L_s^2 (0,T; L_h^2 (\Omega) )}  \nonumber  
\\
  &\le& 
  C  \left( s^2 + h^2 \right) 
  \left( \left\| \Phi \right\|_{L^\infty (0,T; H^4 (\Omega) )}^3 
  + \left\| \Phi_N  \right\|_{L^\infty (0,T; H^4 (\Omega) )}^2 
  \cdot \nrm{\Phi_N}_{H^2 (0,T; H^2)}^2 \right) . 
   \label{truncation-est-2nd-3-13}
\end{eqnarray}

Therefore, the local truncation error estimate for $\tau_1$ is obtained, 
by a combination of (\ref{truncation-est-2nd-2-3}), (\ref{truncation-est-2nd-1-5}), 
(\ref{truncation-est-2nd-4-6}), (\ref{truncation-est-2nd-6}), (\ref{truncation-est-2nd-5-6}), 
(\ref{truncation-est-2nd-3-13})
and a detailed comparison between the truncation equation 
(\ref{truncation-equation-1}), (\ref{truncation-equation-2}) and the original PDE: 
	\begin{eqnarray} 
& & \beta \partial_t^2 \Phi + \partial_t \Phi - \Delta \left( \Phi^3 \right) -  
(1 - \epsilon )  \Delta \Phi - 2   \Delta^2  \Phi - \Delta^3 \Phi   		\nonumber
	\\
& & \hspace{0.2in} = \beta F_{1en} + F_{2en} - F_{3en}  - ( 1 - \epsilon) F_{4en} 
  - F_{5en} - F_{6en} = 0  .  
	\end{eqnarray}
In addition, the constant estimate (\ref{truncation-error}) for $M$ is also satisfied, 
by a careful check. 

The estimate for $\tau_2$ is very similar. We denote the following quantity 
\begin{eqnarray} 
  F_{7e}^{k+1/2} = \left( \partial_t \Phi_N + \frac{s^2}{24} \partial_t^3  \Phi_N\right) 
  (\  \cdot\ , t^{k+1/2} )  .  \label{truncation-est-2nd-7-1}
\end{eqnarray}
A detailed Taylor formula in time gives the following estimate: 
\begin{eqnarray} 
  &&
  F_2^{k+1/2} - F_{7e}^{k+1/2} = \tau_{21}^{k+1/2} ,  
  \mbox{with}  \, \, \, \nonumber 
\\
  &&
  \left\|  \tau_{21}  \right\|_{L_s^2 (0,T)} 
  \le C s^3  \left\|  \Phi_N  \right\|_{H^4 (0,T)}  
   \le C s^3  \left\|  \Phi  \right\|_{H^4 (0,T)}  ,  \label{truncation-est-2nd-7-2}
\end{eqnarray}
at each fixed grid point $(i,j)$. Meanwhile, it is clear that $F_7$ has the following 
decomposition: 
\begin{eqnarray} 
   F_7^{k+1/2}  
   &=& \frac{ \Psi_N^{k+1} + \Psi_N^k}{2} 
   = \frac{ \partial_t \Phi_N^{k+1} + \partial_t \Phi_N^k}{2}   
   - \frac{s^2}{12}  \cdot   \frac{ \partial_t^3 \Phi_N^{k+1} + \partial_t^3 \Phi_N^k}{2}  
   \nonumber 
\\
  &:=& F_{7,1}^{k+1/2} + F_{7,2}^{k+1/2} . 
   \label{truncation-est-2nd-7-3}
\end{eqnarray}
To facilitate the analysis below, we define two more quantities: 
\begin{eqnarray} 
  &&
  F_{7e,1}^{k+1/2} = \left( \partial_t \Phi_N + \frac{s^2}{8} \partial_t^3  \Phi_N\right) 
  (\  \cdot\ , t^{k+1/2} ) ,  \nonumber 
\\
  &&
  F_{7e,2}^{k+1/2} = - \frac{s^2}{12} \partial_t^3  \Phi_N  (\  \cdot\ , t^{k+1/2} )  
   .  \label{truncation-est-2nd-7-4}
\end{eqnarray}
A detailed Taylor formula in time gives the following estimate: 
\begin{eqnarray} 
  &&
  F_{7,1}^{k+1/2} - F_{7e,1}^{k+1/2} = \tau_{22}^{k+1/2} ,   \quad 
  F_{7,2}^{k+1/2} - F_{7e,2}^{k+1/2} = \tau_{23}^{k+1/2}  ,  \quad 
  \mbox{with}  \, \, \, \nonumber 
\\
  &&
  \left\|  \tau_{22}  \right\|_{L_s^2 (0,T)} 
  \le C s^3  \left\|  \Phi_N  \right\|_{H^4 (0,T)}  
   \le C s^3  \left\|  \Phi  \right\|_{H^4 (0,T)}  ,  \label{truncation-est-2nd-7-5}
\\
  &&
  \left\|  \tau_{23}  \right\|_{L_s^2 (0,T)} 
  \le C s^3  \left\|  \Phi_N  \right\|_{H^4 (0,T)}  
   \le C s^3  \left\|  \Phi  \right\|_{H^4 (0,T)}  ,  \label{truncation-est-2nd-7-6}  
\end{eqnarray}
at each fixed grid point $(i,j)$. Consequently, a combination of 
(\ref{truncation-est-2nd-7-2})-(\ref{truncation-est-2nd-7-6}) shows that 
\begin{eqnarray} 
  F_2^{k+1/2} - F_{7}^{k+1/2} = \tau_2^{k+1/2} ,  
  \mbox{with}  \, \, \, 
  \left\|  \tau_{2}  \right\|_{L_s^2 (0,T)}  
   \le C s^3  \left\|  \Phi  \right\|_{H^4 (0,T)}  .  \label{truncation-est-2nd-7-7}
\end{eqnarray}
This in turn implies that 
\begin{eqnarray} 
  \left\| F_2 - F_{7} \right\|_{L_s^2 (0,T; L_h^2)} 
  \le  C s^3  \left\|  \Phi  \right\|_{H^4 (0,T; L^2)}  , \label{truncation-est-2nd-7-8}
\end{eqnarray}
which is exactly (\ref{truncation-equation-2}). Also, the constant $M$ associated with 
$\tau_2$ satisfies (\ref{truncation-error}). The consistency analysis is finished.

	\bibliographystyle{plain}
    \bibliography{mpfc_2nd.bib}

\begin{thebibliography}{10}

\bibitem{backofen07}
R.~Backofen, A.~R\"atz, and A.~Voigt.
\newblock Nucleation and growth by a phase field crystal {(PFC)} model.
\newblock {\em Phil. Mag. Lett.}, 87:813, 2007.

\bibitem{baskaran12}
A.~Baskaran, Z.~Hu, J.S. Lowengrub, C.~Wang, S.M Wise, , and P.~Zhou.
\newblock Energy stable and efficient finite-difference nonlinear multigrid
  schemes for the modified phase field crystal equation.
\newblock {\em J. Comput. Phys.}, 2012.
\newblock in review.

\bibitem{Boyd2001}
J.~Boyd.
\newblock {\em Chebyshev and Fourier Spectram Methods}.
\newblock Dover, New York, second edition, 2001.

\bibitem{CQ1982}
C.~Canuto and A.~Quarteroni.
\newblock Approximation results for orthogonal polynomials in sobolev spaces.
\newblock {\em Math. Comp.}, 38:257--276, 1982.

\bibitem{chen11}
W.~Chen, S.~Conde, C.~Wang, X.~Wang, and S.M. Wise.
\newblock A linear energy stable scheme for a thin film model without slope
  selection.
\newblock {\em J. Sci. Comput.}, 2011.
\newblock (accepted).

\bibitem{cheng08}
M.~Cheng and J.A. Warren.
\newblock An efficient algorithm for solving the phase field crystal model.
\newblock {\em J. Comput. Phys.}, 227:6241, 2008.

\bibitem{elder04}
K.R. Elder and M.~Grant.
\newblock Modeling elastic and plastic deformations in nonequilibrium
  processing using phase filed crystal.
\newblock {\em Phys. Rev. E}, 90:051605, 2004.

\bibitem{elder02}
K.R. Elder, M.~Katakowski, M.~Haataja, and M.~Grant.
\newblock Modeling elasticity in crystal growth.
\newblock {\em Phys. Rev. Lett.}, 88:245701, 2002.

\bibitem{eyre98}
D.~Eyre.
\newblock Unconditionally gradient stable time marching the {C}ahn-{H}illiard
  equation.
\newblock In J.~W. Bullard, R.~Kalia, M.~Stoneham, and L.Q. Chen, editors, {\em
  Computational and Mathematical Models of Microstructural Evolution},
  volume~53, pages 1686--1712, Warrendale, PA, USA, 1998. Materials Research
  Society.

\bibitem{hu09}
Z.~Hu, S.~Wise, C.~Wang, and J.~Lowengrub.
\newblock Stable and efficient finite-difference nonlinear-multigrid schemes
  for the phase-field crystal equation.
\newblock {\em J. Comput. Phys.}, 228:5323--5339, 2009.

\bibitem{HGG2007}
S.~Gottlieb J.~Hesthaven and D.~Gottlieb.
\newblock {\em Spectral Methods for Time-dependent Problems}.
\newblock Cambridge University Press, Cambridge, UK, 2007.

\bibitem{galenko11}
V.~Lebedev, A.~Sysoeva, and P.~Galenko.
\newblock Unconditionally gradient-stable computational schemes in problems of
  fast phase transitions.
\newblock {\em Phys. Rev. E}, 83, 2011.

\bibitem{mellenthin08}
J.~Mellenthin, A.~Karma, and M.~Plapp.
\newblock Phase-field crystal study of grain-boundary premelting.
\newblock {\em Phys. Rev. B}, 78:184110, 2008.

\bibitem{shen11}
J.~Shen, C.~Wang, X.~Wang, and S.M. Wise.
\newblock Second-order convex splitting schemes for gradient flows with
  ehrlich-schwoebel type energy: Application to thin film epitaxy.
\newblock {\em SIAM J. Numer. Anal.}, 50:105--125, 2012.

\bibitem{stefanovic06}
P.~Stefanovic, M.~Haataja, and N.~Provatas.
\newblock Phase-field crystals with elastic interactions.
\newblock {\em Phys. Rev. Lett.}, 96:225504, 2006.

\bibitem{stefanovic09}
Peter Stefanovic, Mikko Haataja, and Nikolas Provatas.
\newblock Phase field crystal study of deformation and plasticity in
  nanocrystalline materials.
\newblock {\em Phys. Rev. E}, 80:046107, 2009.

\bibitem{swift77}
J.~Swift and P.C. Hohenberg.
\newblock Hydrodynamic fluctuations at the convective instability.
\newblock {\em Phys. Rev. A}, 15:319, 1977.

\bibitem{wang10a}
C.~Wang, X.~Wang, and S.M. Wise.
\newblock Unconditionally stable schemes for equations of thin film epitaxy.
\newblock {\em Discrete Cont. Dyn. Sys. Ser. A}, 28:405--423, 2010.

\bibitem{wang11a}
C.~Wang and S.~Wise.
\newblock An energy stable and convergent finite-difference scheme for the
  modified phase field crystal equation.
\newblock {\em SIAM J. Numer. Anal.}, 49:945--969, 2011.

\bibitem{wang10b}
C.~Wang and S.M. Wise.
\newblock Global smooth solutions of the modified phase field crystal equation.
\newblock {\em Methods Appl. Anal.}, 17:191--212, 2010.

\bibitem{wise10}
S.M. Wise.
\newblock Unconditionally stable finite difference, nonlinear multigrid
  simulation of the {C}ahn-{H}illiard-{H}ele-{S}haw system of equations.
\newblock {\em J. Sci. Comput.}, 44:38--68, 2010.

\bibitem{wise09}
S.M. Wise, C.~Wang, and J.S. Lowengrub.
\newblock An energy stable and convergent finite-difference scheme for the
  phase field crystal equation.
\newblock {\em SIAM J. Numer. Anal.}, 47:2269--2288, 2009.

\end{thebibliography}
    
	\end{document}